\documentclass[leqno,12pt]{amsart} %leqno is the option to put formula numbers on the left side
\usepackage{amssymb}
\theoremstyle{plain} %text of this environment is typesetted in italics
\newtheorem{theorem}{\indent\sc Theorem}[section]
\newtheorem{lemma}[theorem]{\indent\sc Lemma}

\newtheorem{proposition}[theorem]{\indent\sc Proposition}

\theoremstyle{definition} %text of this environment is typesetted in roman letters
\newtheorem{definition}[theorem]{\indent\sc Definition}
\newtheorem{remark}[theorem]{\indent\sc Remark}
\newtheorem{example}[theorem]{\indent\sc Example}

\begin{document}

\title[Constant terms of Eisenstein series over a totally real field]{Constant terms of Eisenstein series \\ over a totally real field} %title of paper and the running head option

\author[T. Ozawa]{Tomomi Ozawa} %first author's name and the running head option

%\dedicatory{Dedicated to Professor Xxx Yyy on his sixtieth birthday}
%%%%%%%%%%%%%%% footnote %%%%%%%%%%%%%%%%
%In case \subjclass[2010] command is not effective
%(or the version of amsart.cls is old), write as follows:
%\renewcommand{\thefootnote}{\fnsymbol{footnote}}
%\footnote[0]{2010\textit{ Mathematics Subject Classification}.
%Primary 00; Secondary 00.}
%
%\thanks{ % Thanks
%$^*$Partly supported by the Grant-in-Aid for Scientific Research (A),
%Japan Society for the Promotion of Science.
%}
%%%%%%%%%%%% Authors addresses %%%%%%%%%%%%%
\address{% First Author
Mathematical Institute, Graduate School of Science, Tohoku University \\
6-3 Aramaki Aza-Aoba, Aoba-ku, Sendai 980-8578 \\
Japan
}
\email{sb2m06@math.tohoku.ac.jp}
%%%%%%%%%%%%%%%%%%%%%%%%%%%%%%%%%%%%%%%%%

\maketitle

\begin{abstract}
In this paper, we compute constant terms of Eisenstein series defined over a totally real field at all cusps. We explicitly describe the constant terms of Eisenstein series at each equivalence class of cusps in terms of special values of Hecke $L$-functions. This investigation is motivated by M. Ohta's work on congruence modules related to Eisenstein series defined over the field of rational numbers. 
\end{abstract}

\tableofcontents

\section*{Introduction} %delete * to number this section
In this paper, we compute constant terms of Eisenstein series defined over a totally real field, at various cusps. In his paper \cite{O} published in 2003, M. Ohta computed the constant terms of Eisenstein series of weight $2$ over ${\mathbb{Q}}$, at all equivalence classes of cusps. As for Eisenstein series defined over a totally real field, S. Dasgupta, H. Darmon and R. Pollack calculated the constant terms at particular (not all) equivalence classes of cusps in 2011 \cite{DDP}. We compute constant terms of Eisenstein series defined over a general totally real field at \emph{all} equivalence classes of cusps, and describe them  explicitly in terms of Hecke $L$-functions. 

This investigation is motivated by Ohta's work \cite{O} on congruence modules related to Eisenstein series defined over ${\mathbb{Q}}$. Let $p \geq 5$ be a prime number. In his paper \cite{O} Ohta defined and computed congruence modules attached to a pair of an exact sequence of $\Lambda$-adic forms and its Hecke-equivariant splitting over the fractional field of $\Lambda$, where $\Lambda$ is the Iwasawa algebra $\Lambda={\mathbb{Z}}_p[[1+p{\mathbb{Z}}_p]]$. In his computation, the constant terms of Eisenstein series over $\mathbb{Q}$ at all equivalence classes of cusps of $\Gamma_1(Np^r)$ are necessary. The congruence module which he considered can be described by using Kubota-Leopoldt $p$-adic $L$-functions, which is obtained by $p$-adic interpolation of a family of Dirichlet $L$-functions. As is well known, the constant terms of Eisenstein series defined over ${\mathbb{Q}}$ can be expressed in terms of special values of Dirichlet $L$-functions. Indeed the congruence module above was computed by using values of constant terms obtained from a $p$-adic analytic family of Eisenstein series. 

It is expected to extend Ohta's work to the case of totally real fields, since his result has been applied to several important problems in number theory, such as problems concerning $\mu$-invariant of the cyclotomic ${\mathbb{Z}}_p$-extension over a number field (see \cite{I} for details). 

We explain the main result of this paper. We consider Hilbert modular forms defined over a totally real field $F$ of degree $d$. Let $O$ be the ring of integers of $F$, $I$ the set of the embeddings of $F$ into $\mathbb{R}$, ${\mathfrak d}$ the different of $F$, ${\rm Cl}_F$ the ideal class group of $F$, and ${\rm Cl}_F^+$ the narrow ideal class group of $F$. For a vector $q=(q_{\sigma})_{\sigma \in I} \in ({\mathbb{Z}}/2{\mathbb{Z}})^d$ and $a \in F$, we write ${\rm sgn}(a)^q=\prod_{\sigma \in I} {\rm sgn}(a^{\sigma})^{q_\sigma}$. Let ${\mathfrak H}$ be the upper half plane ${\mathfrak H}=\left\{z \in {\mathbb{C}} \mid {\rm Im}(z)>0 \right\}$, $GL_2(F)$ the group of all $2 \times 2$ invertible matrices with coefficients in $F$, and $GL_2^+(F)$ the subgroup of $GL_2(F)$ consisting of all matrices with totally positive determinant. For a function $f$ on ${\mathfrak H}^I$, a matrix $\gamma=\left(
\begin{smallmatrix}
 a & b \\
 c & d
\end{smallmatrix}
\right) \in GL_2^+(F)$, and an integer $k \geq 0$, we define a new function $f|_k \gamma$ on ${\mathfrak H}^I$ by the rule
\begin{align*}
 (f|_k \gamma)(z) &= \det(\gamma)^{\frac{k}{2}}(cz+d)^{-k}f(\gamma z)
\end{align*}
where
$$\det(\gamma)=\prod_{\sigma \in I}\det(\gamma)^{\sigma}, \ cz+d=\prod_{\sigma \in I}(c^{\sigma}z_{\sigma}+d^{\sigma}), \ \text{and} \ \gamma z=\left(\frac{a^{\sigma}z_{\sigma}+b^{\sigma}}{c^{\sigma}z_{\sigma}+d^{\sigma}} \right)_{\sigma \in I}. $$
Let ${\mathfrak b}$ be an integral ideal of $F$, and $\psi$ a narrow ray class character modulo ${\mathfrak b}$ and signature $r \in ({\mathbb{Z}}/2{\mathbb{Z}})^d$. Then we have a well-defined finite order character $\psi_f : (O/{\mathfrak b})^\times \rightarrow {\mathbb{C}}^\times$ given by $\psi_f(a)=\psi(aO){\rm sgn}(a)^r$. For each $\lambda \in {\rm Cl}_F^+$, we choose a representative fractional ideal ${\mathfrak t}_\lambda$ of $\lambda$ and put
\begin{align*}
 \Gamma_\lambda({\mathfrak b}) &= \left\{\left(
\begin{array}{cc}
 a & b \\
 c & d
\end{array}
\right) \in GL_2^+(F) \bigg\vert a, d \in O, b \in {\mathfrak d}^{-1}{\mathfrak t}_\lambda^{-1}, c \in {\mathfrak b}{\mathfrak d}{\mathfrak t}_\lambda \ \text{and} \ ad-bc \in O^\times \right\}.
\end{align*}
\begin{definition}
\label{defn:001}
Let $k, {\mathfrak b}$ and $\psi$ be as before, and $h=\# {\rm Cl}_F^+$. An $h$-tuple $f=(f_\lambda)_{\lambda \in {\rm Cl}_F^+}$ of holomorphic functions $f_\lambda: {\mathfrak H}^I \rightarrow \mathbb{C}$ is a Hilbert modular form of (parallel) weight $k$, level ${\mathfrak b}$ and character $\psi$ if for each $\lambda \in {\rm Cl}_F^+$, 
\begin{align*}
 f_\lambda |_k \gamma &= \psi_f(d) f_\lambda \ \text{for all} \ \gamma=\left(
\begin{array}{cc}
 * & * \\
 * & d
\end{array}
\right) \in \Gamma_\lambda({\mathfrak b}). 
\end{align*}
When $F={\mathbb{Q}}$, $f$ has to be holomorphic around each cusp; that is, for any $\gamma \in SL_2({\mathbb{Z}})$, 
\begin{align*}
 (f|_k \gamma)(z) &= \sum_{n=0}^\infty a\left(\frac{n}{N}, f \right)q^{\frac{n}{N}} \ \ (q=\exp(2\pi i z))
\end{align*}
where $N$ is the positive integer determined by $N{\mathbb{Z}}={\mathfrak b}$. 
\end{definition}
There is an analogue of Eisenstein series defined over ${\mathbb{Q}}$, in the following sense: 
\begin{proposition}[Shimura, \cite{S}] 
\label{prop:002}
For narrow ray class characters $\eta$ and $\psi$ of conductors ${\mathfrak a}, {\mathfrak b}$ and signatures $q, r \in ({\mathbb{Z}}/2{\mathbb{Z}})^d$ respectively, and an integer $k \geq 1$ with
\begin{align*}
 q+r & \equiv (k, k, \ldots, k) \bmod (2{\mathbb{Z}})^d,
\end{align*}
there exists a Hilbert modular form $E_k(\eta, \psi)$ of weight $k$, level ${\mathfrak a}{\mathfrak b}$ and character $\eta\psi$ such that its normalized coefficients satisfy
\begin{align}
 c({\mathfrak n}, E_k(\eta, \psi)) &= \sum_{{\mathfrak n}_1 \mid {\mathfrak n}} \eta\left(\frac{{\mathfrak n}}{{\mathfrak n}_1} \right) \psi({\mathfrak n}_1) {\rm N}({\mathfrak n}_1)^{k-1} \ \text{for each non-zero integral ideal} \ {\mathfrak n}; \ \text{and} \label{eqn:001} \\
 c_\lambda(0, E_k(\eta, \psi)) &= \left\{
\begin{array}{l}
 \delta_{\eta, {\rm id}}2^{-d}L(\psi, 1-k) \ \mbox{if\ } k \geq 2, \\
 2^{-d}(\delta_{\eta, {\rm id}}L(\psi, 0)+\delta_{\psi, {\rm id}}L(\eta, 0)) \ \mbox{if\ } k=1 
\end{array}
\right. \ \text{for each} \ \lambda \in {\rm Cl}_F^+. \label{eqn:002}
\end{align}
The sum in $(\ref{eqn:001})$ runs over all integral ideals ${\mathfrak n}_1$ dividing ${\mathfrak n}$, and ${\rm N}={\rm N}_{F/{\mathbb{Q}}}$ is the norm of $F/{\mathbb{Q}}$. In $(\ref{eqn:002})$, $\delta_{\eta, {\rm id}}=1$ if $\eta={\rm id}$ $($i.e., ${\mathfrak a}=O)$ and $0$ otherwise. $L(\eta, s)$ denotes the Hecke $L$-function attached to $\eta$ $($we use the same notation for other characters$)$. We call $E_k(\eta, \psi)$ the Eisenstein series of weight $k$ associated with characters $(\eta, \psi)$. 
\end{proposition}

Our main result is the computation of the constant terms of $E_k(\eta, \psi)$ at \emph{all} equivalence classes of cusps. We write $E_k(\eta, \psi)=(E_k(\eta, \psi)_\lambda)_{\lambda \in {\rm Cl}_F^+}$ in the sense of Definition \ref{defn:001}. The following computation is the first step of our main result: 
\begin{proposition}
\label{prop:003}
Let $\Gamma_\lambda^1(O)=\Gamma_\lambda(O) \cap SL_2(F)$. For a matrix
\begin{align*}
 A_\lambda &= \left(
\begin{array}{cc}
 \alpha_\lambda & \beta_\lambda \\
 \gamma_\lambda & \delta_\lambda
\end{array}
\right) \in \Gamma_\lambda^1(O)
\end{align*}
$($in particular $\gamma_\lambda \in {\mathfrak d}{\mathfrak t}_\lambda)$, the constant term of ${\rm N}({\mathfrak t}_\lambda)^{-\frac{k}{2}}E_k(\eta, \psi)_\lambda|_k A_\lambda$ is equal to
\begin{align*}
 & \frac{1}{2^d}\frac{\tau(\eta \psi^{-1})}{\tau(\psi^{-1})} \left(\frac{{\rm N}({\mathfrak b})}{{\rm N}({\mathfrak c})} \right)^k {\rm sgn}(-\gamma_\lambda)^q \eta(\gamma_\lambda ({\mathfrak b}{\mathfrak d}{\mathfrak t}_\lambda)^{-1}) {\rm sgn}(\alpha_\lambda)^r \psi^{-1}(\alpha_\lambda) \\
 & \ \ \times L(\eta^{-1} \psi, 1-k) \prod_{{\mathfrak q} \mid {\mathfrak m}, \ {\mathfrak q} \nmid {\mathfrak c}} (1-\eta \psi^{-1}({\mathfrak q}){\rm N}({\mathfrak q})^{-k})
\end{align*}
if $\gamma_\lambda \in {\mathfrak b}{\mathfrak d}{\mathfrak t}_\lambda$, and $0$ otherwise.
Here $\tau(\psi^{-1})$ is the Gauss sum of $\psi^{-1}$, ${\mathfrak c}={\rm cond}(\eta^{-1}\psi)$, ${\mathfrak m}={\mathfrak a}{\mathfrak b}$, and the last product runs over all prime ideals ${\mathfrak q}$ with ${\mathfrak q} \mid {\mathfrak m}$ and ${\mathfrak q} \nmid {\mathfrak c}$. 
\end{proposition}
This proposition implies Ohta's computation of the constant terms of Eisenstein series over ${\mathbb{Q}}$ in \cite{O}, as well as Dasgupta, Darmon and Pollack's description of the constant terms over a totally real field in \cite{DDP} (see Remark \ref{rmk:305} of Section 3.1 for details). 

According to Proposition \ref{prop:003}, the second thing we should investigate is the equivalence classes of cusps of $\Gamma_\lambda^1(O)$. Let $B^1(F)$ be the group of all upper triangular matrices of $SL_2(F)$. It is well known that there is a bijection
$$SL_2(F)/B^1(F) \rightarrow {\mathbb{P}}^1(F); \ \gamma \mapsto \gamma(\infty).$$
Thus we need to describe the set $\Gamma_\lambda^1(O) \backslash SL_2(F)/B^1(F)$ explicitly. We will show in Proposition \ref{prop:308} that the map
\begin{align}
 il_{{\mathfrak t}_\lambda^{-1}} : \Gamma_\lambda^1(O) \backslash SL_2(F)/B^1(F) \rightarrow {\rm Cl}_F & ; \ \left(
\begin{array}{cc}
 a & b \\
 c & d
\end{array}
\right) \mapsto c({\mathfrak d}{\mathfrak t}_\lambda)^{-1}+aO \label{eqn:003}
\end{align}
is a bijection. We will compute in Section 3.3 a suitable representative under $il_{{\mathfrak t}_\lambda^{-1}}$ of a given ideal class in ${\rm Cl}_F$ represented by a non-zero integral ideal ${\mathfrak r}_0$ explicitly. We can find a matrix
\begin{align*}
 A_\lambda &= \left(
\begin{array}{cc}
 \alpha_\lambda & \beta_\lambda \\
 \gamma_\lambda & \delta_\lambda
\end{array}
\right) \in SL_2(F)
\end{align*}
with $il_{{\mathfrak t}_\lambda^{-1}}(A_\lambda)={\mathfrak r}_0$ so that 
$$\alpha_\lambda O={\mathfrak n}_2{\mathfrak r}_0, \ \beta_\lambda \in ({\mathfrak d}{\mathfrak t}_\lambda{\mathfrak r}_0)^{-1}, \ \gamma_\lambda O={\mathfrak n}_1{\mathfrak d}{\mathfrak t}_\lambda{\mathfrak r}_0 \ \text{and} \ \delta_\lambda \in {\mathfrak r}_0^{-1}.$$
Here ${\mathfrak n}_i \ (i=1, 2)$ are integral ideals satisfying ${\mathfrak n}_1+{\mathfrak n}_2=O$ and ${\mathfrak n}_1+{\mathfrak b}=O$ (see Proposition \ref{prop:309} for details). 
\begin{theorem}
\label{thm:004}
Let ${\mathfrak r}_0, A_\lambda$, and ${\mathfrak n}_i \ (i=1, 2)$ be as above. If ${\mathfrak r}_0$ is not equivalent to the class of $\infty$ under the bijection $(\ref{eqn:003})$, the constant term of ${\rm N}({\mathfrak t}_\lambda)^{-\frac{k}{2}} E_k(\eta, \psi)_\lambda$ at the cusp corresponding to the equivalence class of ${\mathfrak r}_0$ $($i.e., the constant term of ${\rm N}({\mathfrak t}_\lambda)^{-\frac{k}{2}} E_k(\eta, \psi)_\lambda|A_\lambda$ at $\infty)$ is equal to
\begin{align*}
 & \delta_{\psi, {\rm id}}\frac{1}{2^d}\tau(\eta)\left(\frac{{\rm N}({\mathfrak r}_0)}{{\rm N}({\mathfrak a})} \right)^k {\rm sgn}(-\gamma_\lambda)^q \eta({\mathfrak n}_1) L(\eta^{-1}, 1-k).
\end{align*}
\end{theorem}
%The author thinks that quite a few number theorists have explicitly carried out by their own hands the computation of constant terms of Eisenstein series over a totally real field (actually the author knows some researchers to whom this case applies). However it seems that no detailed computation had ever been published. We are going to write down every detail of the computation in Section 3, though the computation is a little bit lengthy. 

\textbf{Layout.} Section 1 is devoted to reviewing elliptic modular forms, Eisenstein series and the computation of Fourier expansion. In Section 2, we give basics of Hilbert modular forms and Eisenstein series. In the last section, we investigate the equivalence classes of cusps of certain congruence subgroups, and compute the constant terms of Eisenstein series at all equivalence classes of cusps. 

\textbf{Notation.} Throughout this paper we use the following notation: 
\begin{itemize}
\item $i \in {\mathbb{C}}$: a fixed square root of $-1$; 
\item ${\mathfrak H}$: the upper half plane ${\mathfrak H}=\left\{z \in {\mathbb{C}} \mid {\rm Im}(z)>0 \right\}$; 
\item $\infty=\displaystyle \lim_{t \to +\infty}it$: the point at infinity; 
\item $GL_2({\mathbb{R}})$: the group of all $2 \times 2$ invertible matrices with real coefficients; 
\item $GL_2^+({\mathbb{R}})$: the subgroup of $GL_2({\mathbb{R}})$ consisting of all matrices with positive determinant. 
\end{itemize}

\textbf{Acknowledgement.} The author would like to express her hearty thanks to her supervisor, Professor Nobuo Tsuzuki, for his helpful advice and unceasing encouragement. She also thanks Professors Masataka Chida, Ming-Lun Hsieh, Atsushi Yamagami, and Dr. Yuichi Hirano for useful information, suggestions and discussions. 

%%%%%%%%%%%%%%%%%%%%%%%%%%%%%%%%%%%%%%

\section{Elliptic modular forms}
In Section 1, we recall elliptic modular forms and Hecke operators. We give few proofs in this section. We refer to \cite{H1} Chapter 5 and \cite{M} Chapters 4 and 7 for details. 
\subsection{Basic definitions and examples}
In this subsection, we review basic definitions and examples of elliptic modular forms. Let $SL_2({\mathbb{Z}})$ denote the group of all matrices of determinant $1$ with coefficients in ${\mathbb{Z}}$. For each positive integer $N$, we define three congruence subgroups of $SL_2({\mathbb{Z}})$ (a subgroup $\Gamma$ of $SL_2({\mathbb{Z}})$ is called a congruence subgroup if $\Gamma$ contains $\Gamma(N)$ below for some $N \in {\mathbb{Z}}_{\geq 1}$): 
\begin{align*}
 \Gamma_0(N) &= \left\{\left(
\begin{array}{cc}
 a & b \\
 c & d
\end{array}
\right) \in SL_2({\mathbb{Z}}) \bigg\vert c \equiv 0 \bmod N \right\}, \\
 \Gamma_1(N) &= \left\{\left(
\begin{array}{cc}
 a & b \\
 c & d
\end{array}
\right) \in \Gamma_0(N) \bigg\vert a \equiv d \equiv 1 \bmod N \right\}, \ \text{and} \\
 \Gamma(N) &= \left\{\left(
\begin{array}{cc}
 a & b \\
 c & d
\end{array}
\right) \in \Gamma_1(N) \bigg\vert b \equiv 0 \bmod N \right\}. 
\end{align*}
All three groups are of finite index in $SL_2({\mathbb{Z}})$, since there is an isomorphism $SL_2({\mathbb{Z}})/\Gamma(N) \cong SL_2({\mathbb{Z}}/N{\mathbb{Z}})$. 
Note that $\Gamma(N)$ is normal in $SL_2({\mathbb{Z}})$ for any $N \geq 1$, while $\Gamma_0(N)$ and $\Gamma_1(N)$ are not normal in $SL_2({\mathbb{Z}})$ when $N \geq 2$. 

We now explain the notion of elliptic modular forms. For a function $f: {\mathfrak H} \rightarrow {\mathbb{C}}$, a matrix $\gamma=\left(
\begin{smallmatrix}
 a & b \\
 c & d
\end{smallmatrix}
\right) \in GL_2^+({\mathbb{R}})$ and an integer $k \in \mathbb{Z}$, we define another function $f|_k \gamma$ on ${\mathfrak H}$ by the rule
\begin{align*}
 (f|_k \gamma)(z) &= {\rm det}(\gamma)^{\frac{k}{2}}(cz+d)^{-k}f(\gamma z) 
\end{align*}
where $\gamma z=\frac{az+b}{cz+d}$. Note that $\gamma z \in {\mathfrak H}$, since ${\rm det}(\gamma)>0$. If $f$ is holomorphic on ${\mathfrak H}$, so is $f|_k \gamma$. 

From now on let $k$ be a non-negative integer. Any holomorphic function $f: {\mathfrak H} \rightarrow {\mathbb{C}}$ with $f|_k \gamma =f$ for all $\gamma \in \Gamma(N)$ can be expanded as follows:
\begin{align}
 f(z) &= \sum_{n \in {\mathbb{Z}}} a\left(\frac{n}{N}, f \right)q^{\frac{n}{N}} \ (q={\rm exp}(2 \pi i z)) . \label{eqn:101}
\end{align}
We call (\ref{eqn:101}) the Fourier expansion of $f$. 
\begin{definition}
\label{defn:101}
Let $\Gamma$ be a subgroup of $SL_2({\mathbb{Z}})$ containing $\Gamma(N)$. The space $M_k(\Gamma)$ of (elliptic) modular forms of weight $k$ and level $\Gamma$ consists of elements $f$ such that
\begin{itemize}
\item $f$ is a holomorphic function on ${\mathfrak H}$; 
\item $f|_k \gamma =f$ for all $\gamma \in \Gamma$; 
\item $a(n/N, f|_k \gamma)=0$ for all integers $n<0$ and $\gamma \in SL_2({\mathbb{Z}})$.
\end{itemize}
Note that for each $\gamma \in SL_2({\mathbb{Z}})$, we have $(f|_k \gamma)|_k \delta=f|_k \gamma$ for all $\delta \in \Gamma(N)$ since $\Gamma(N)$ is normal in $SL_2({\mathbb{Z}})$, and thus it makes sense to consider the coefficients $a(n/N, f|_k \gamma)$. We often omit the subscript $k$ of $f|_k \gamma$ when there is no ambiguity concerning weight. 
\end{definition}
Each $\gamma=\left(
\begin{smallmatrix}
 a & b \\
 c & d
\end{smallmatrix}
\right) \in \Gamma_0(N)$ induces an automorphism
$$R_d: M_k(\Gamma_1(N)) \rightarrow M_k(\Gamma_1(N)); \ f \mapsto f|_k \gamma$$
on $M_k(\Gamma_1(N))$. Since we know that
$$\Gamma_0(N)/\Gamma_1(N) \cong ({\mathbb{Z}}/N{\mathbb{Z}})^{\times}; \ \left(
\begin{array}{cc}
 a & b \\
 c & d
\end{array}
\right) \mapsto d \bmod N,$$
$R_d$ depends only on $d \bmod N$. In other words, $({\mathbb{Z}}/N{\mathbb{Z}})^{\times}$ acts on $M_k(\Gamma_1(N))$ via $R_d$. We call $R_d$ the diamond operator attached to $d \bmod N$. For a Dirichlet character $\chi$ defined modulo $N$, we define the $\chi$-eigenspace
\begin{align*}
 M_k(\Gamma_0(N), \chi) &= \left\{f \in M_k(\Gamma_1(N)) \mid f|R_d=\chi(d)f \ \text{for all} \ d \in ({\mathbb{Z}}/N{\mathbb{Z}})^{\times} \right\}
\end{align*}
of $M_k(\Gamma_1(N))$. Then $M_k(\Gamma_1(N))$ can be decomposed as follows:
$$M_k(\Gamma_1(N))=\bigoplus_{\chi} M_k(\Gamma_0(N), \chi); \ f \mapsto \left(\frac{1}{\varphi(N)}\sum_{d \in ({\mathbb{Z}}/N{\mathbb{Z}})^{\times}}\chi^{-1}(d)f|R_d \right)_{\chi}.$$
Here $\chi$ runs over all Dirichlet characters modulo $N$ and $\varphi(N)=\# ({\mathbb{Z}}/N{\mathbb{Z}})^{\times}$ is the Euler function. We call an element of $M_k(\Gamma_0(N), \chi)$ a modular form of weight $k$, level $\Gamma_0(N)$ and character (also called Nebentypus) $\chi$. 
\begin{remark}
\label{rmk:102}
For $f \in M_k(\Gamma_0(N), \chi)$, we have
$$\chi(-1)f=f|_k \left(
\begin{array}{cc}
 -1 & 0 \\
 0 & -1
\end{array}
\right)=(-1)^k f.$$
Thus we see that $f=0$ when $\chi(-1)=-(-1)^k$. From now on, we assume $\chi(-1)=(-1)^k$ (i.e., $\chi$ and $k$ have the same parity) whenever we consider modular forms of weight $k$ and character $\chi$. When $N=1$, the above equality implies that $M_k(SL_2({\mathbb{Z}}))=0$ if $k$ is odd.
\end{remark}

\begin{example}
\label{exam:103}
Eisenstein series are one of the most basic and important examples of modular forms, and are also the main object of this paper. For simplicity we assume $k \geq 2$. Let $\eta$ (resp. $\psi$) be a Dirichlet character of conductor $u$ (resp. $v$), and $(\eta \psi)(-1)=(-1)^k$. We put $N=uv$. Consider first the case when $k \geq 3$. We let
\begin{align*}
 E_k ' (\eta, \psi)(z) &= \sum_{a_1=1}^u \sum_{a_2=1}^N \eta(a_1)\psi^{-1}(a_2)G_k(z, a_1v, a_2, N)  
\end{align*}
where for integers $a_i \ (i=1, 2)$, we define
\begin{align*}
 G_k(z, a_1, a_2, N) &= \sum_{\substack{
 (a, b) \in {\mathbb{Z}}^2, \ (a, b) \neq (0, 0), \\
 a \equiv a_1 \bmod N, \ b \equiv a_2 \bmod N
}} \frac{1}{(az+b)^k} 
\end{align*}
(which is absolutely convergent on ${\mathfrak H}$). We note that $E_k '(\eta, \psi)$ here is what is written as $E_k '(\theta, \psi)$ in \cite{O}. Our $\psi$ corresponds to $\theta$ there and $\eta$ here corresponds to $\psi$ there. We use this notation in order to be consistent with notation of Hilbert Eisenstein series in Sections 2 and 3. For each $\gamma=\left(
\begin{smallmatrix}
 a & b \\
 c & d
\end{smallmatrix}
\right) \in SL_2({\mathbb{Z}})$, $G_k(z, a_1, a_2, N)$ satisfies the transformation law
\begin{align*}
 G_k(z, a_1, a_2, N)|_k \gamma &= G_k(z, a_1a+a_2c, a_1b+a_2d, N) 
\end{align*}
and hence we have $E_k '(\eta, \psi)|_k \gamma=(\eta \psi)(d)E_k '(\eta, \psi)$ for any $\gamma=\left(
\begin{smallmatrix}
 a & b \\
 c & d
\end{smallmatrix}
\right) \in \Gamma_0(N)$. When $k \geq 3$, $G_k(z, a_1, a_2, N)$ is expanded as
\begin{align*}
 G_k(z, a_1, a_2, N) &= \delta(a_1, N)\sum_{\substack{
 b \equiv a_2 \bmod N, \\
 b \neq 0
}} b^{-k} \\
 & \ \ +\frac{(-2 \pi i)^k}{N^k (k-1)!} \sum_{\substack{
 a \equiv a_1 \bmod N, \\
 a>0
}} \sum_{m=1}^\infty m^{k-1} {\mathbf e}\left(\frac{ma_2}{N} \right){\mathbf e}\left(\frac{maz}{N} \right) \\
 & \ \ \ +\frac{(-2 \pi i)^k}{N^k (k-1)!}(-1)^k \sum_{\substack{
 a \equiv a_1 \bmod N, \\
 a<0
}} \sum_{m=1}^\infty m^{k-1} {\mathbf e}\left(-\frac{ma_2}{N} \right){\mathbf e}\left(-\frac{maz}{N} \right).
\end{align*}
Here ${\mathbf e}(x)={\rm exp}(2\pi ix)$, and $\delta(a_1, N)=1$ if $a_1 \in N{\mathbb{Z}}$ and $0$ otherwise. Consequently the Fourier expansion of $E_k '(\eta, \psi)$ is
\begin{align*}
 E_k '(\eta, \psi)(z) &= \frac{2(2\pi i)^k \tau(\psi^{-1})}{v^k(k-1)!} \\
 & \ \ \times \left\{\delta_{\eta, {\rm id}}\psi(-1)2^{-1}L(\psi, 1-k)+(-1)^k\sum_{n=1}^\infty \left(\sum_{0<t \mid n}\eta\left(\frac{n}{t} \right)\psi(t)t^{k-1} \right)q^n \right\}. 
\end{align*}
Here $\delta_{\eta, {\rm id}}=1$ if $\eta$ is trivial (i.e., $u=1$) and $0$ otherwise. $\tau(\psi^{-1})$ is the usual Gauss sum
\begin{align*}
 \tau(\psi^{-1}) &= \sum_{m=1}^v \psi^{-1}(m){\mathbf e}\left(\frac{m}{v} \right)
\end{align*}
of $\psi^{-1}$ and $L(\psi, 1-k)$ denotes the Dirichlet $L$-function attached to $\psi$. Thus we have $E_k '(\eta, \psi) \in M_k(\Gamma_0(N), \eta \psi)$. We normalize $E_k '(\eta, \psi)$ and define
\begin{align*}
 E_k(\eta, \psi)(z) &= \frac{v^{k-1}\tau(\psi)(k-1)!}{2(2\pi i)^k}E_k '(\eta, \psi)(z) \\ 
 &= \delta_{\eta, {\rm id}}2^{-1}L(\psi, 1-k)+\sum_{n=1}^\infty \left(\sum_{0<t \mid n}\eta\left(\frac{n}{t} \right)\psi(t)t^{k-1} \right)q^n.
\end{align*}
We call $E_k(\eta, \psi)$ the Eisenstein series of weight $k$ associated with characters $(\eta, \psi)$. Note that when $\eta=\psi={\rm id}$ (i.e., $u=v=N=1$) we have $E_k(\eta, \psi)=0$ for any odd $k \geq 3$. 

When $k=2$, we need a slight modification. For integers $a_i \ (i=1, 2)$, the series
\begin{align*}
 G_2(z, s, a_1, a_2, N) &= \sum_{\substack{
 (a, b) \in {\mathbb{Z}}^2, \ (a, b) \neq (0, 0), \\
 a \equiv a_1 \bmod N, \ b \equiv a_2 \bmod N
}} \frac{1}{(az+b)^2|az+b|^{2s}} 
\end{align*}
(which is absolutely convergent on ${\rm Re}(s)>0$) has a meromorphic continuation in $s$ to the whole complex plane, and is holomorphic at $s=0$. We put $G_2(z, a_1, a_2, N)=G_2(z, 0, a_1, a_2, N)$. The Fourier expansion of $G_2(z, a_1, a_2, N)$ is
\begin{align*}
 G_2(z, a_1, a_2, N) &= \delta(a_1, N)\sum_{\substack{
 b \equiv a_2 \bmod N, \\
 b \neq 0
}} b^{-2}-\frac{\pi}{N^2{\rm Im}(z)} \\
 & \ -\frac{4\pi^2}{N^2}\sum_{a \equiv a_1 \bmod N} \sum_{\substack{
 m \in {\mathbb{Z}}, \\
 ma>0
}} |m| {\mathbf e}\left(\frac{ma_2}{N} \right){\mathbf e}\left(\frac{maz}{N} \right).
\end{align*}
In particular $G_2(z, a_1, a_2, N)$ is not holomorphic at $\infty$. If at least one of $\eta$ or $\psi$ is non-trivial, the sum
\begin{align*}
 E_2 '(\eta, \psi)(z) &= \sum_{a_1=1}^u \sum_{a_2=1}^N \eta(a_1)\psi^{-1}(a_2) G_2(z, a_1v, a_2, N)
\end{align*}
is holomorphic at $\infty$ by the orthogonal relation of a Dirichlet character and the weight $2$ normalized Eisenstein series
\begin{align*}
 E_2(\eta, \psi)(z) &= \frac{v \tau(\psi)}{2(2\pi i)^2}E_2'(\eta, \psi)(z) \\
 &= \delta_{\eta, {\rm id}}2^{-1}L(\psi, -1)+\sum_{n=1}^\infty \left(\sum_{0<t \mid n} \eta\left(\frac{n}{t} \right)\psi(t)t \right)q^n
\end{align*}
belongs to $M_2(\Gamma_0(N), \eta \psi)$. However when $\eta=\psi={\rm id}$, $E_2'(z)=G_2(z, 1, 1, 1)$ is not holomorphic at $\infty$. As is well known, the holomorphic series
\begin{align*}
 E_2(z) &= \frac{3}{{\pi}^2}\left(E_2 '(z)+\frac{\pi}{{\rm Im}(z)} \right)=1-24 \sum_{n=1}^\infty \left(\sum_{0<t \mid n} t \right)q^n.
\end{align*}
is not a modular form (for details we refer to \cite{Og} Section 3). 
\end{example}

%%%%%%%%%%%%%%%%%%%%%%%%%%%%%%%%%%%%%%%%%%%%

\section{Hilbert modular forms}
In Section 2, we first recall the definitions and basic properties of Hilbert modular forms, and the Eisenstein series constructed by Shimura in \cite{S}. Section 2 is based on \cite{DDP} Section 2, \cite{H1} Chapter 9, and \cite{S}. Throughout Sections 2 and 3, we use the following notation:
\begin{itemize}
\item $F$: a totally real number field of degree $d$; 
\item $O$: the ring of integers of $F$; 
\item $I$: the set of the embeddings of $F$ into $\mathbb{R}$;  
\item $F_{+}$: the set of the totally positive elements of $F$; 
\item $GL_2(F)$: the group of all $2 \times 2$ invertible matrices with coefficients in $F$; 
\item $GL_2^{+}(F)$: the subgroup of $GL_2(F)$ consisting of all matrices with determinant in $F_{+}$; 
\item $SL_2(F)$: the subgroup of $GL_2^+(F)$ consisting of all matrices with determinant $1$;  
\item $\mathfrak d$: the different of $F/{\mathbb{Q}}$; 
\item ${\rm N}={\rm N}_{F/{\mathbb{Q}}}$: the norm of $F/{\mathbb{Q}}$; 
\item For a vector $r=(r_{\sigma})_{\sigma \in I} \in ({\mathbb{Z}}/2{\mathbb{Z}})^d$ and $a \in F$, we write ${\rm sgn}(a)^r=\prod_{\sigma \in I} {\rm sgn}(a^{\sigma})^{r_\sigma}$. 
\end{itemize}

\subsection{Definitions and basic properties}
We begin by recalling the definition of narrow ray class characters of $F$. Let $\mathfrak b$ be a non-zero integral ideal of $F$. We put
\begin{align*}
 I({\mathfrak b}) &= \left\{\frac{{\mathfrak n}}{{\mathfrak c}} \bigg\vert {\mathfrak n} \ \text{and} \ {\mathfrak c} \ \text{are integral ideals and prime to} \ {\mathfrak b} \right\}, \\
 P_+ &= \left\{aO \mid a \in F_+ \right\}, \ \text{and} \\
 P_+({\mathfrak b}) &= P_+ \cap \left\{aO \mid a \equiv 1\bmod ^{\times} {\mathfrak b} \right\},
\end{align*}
where $a \equiv 1 \bmod ^{\times} {\mathfrak b}$ means that $aO \in I({\mathfrak b})$ and there exists $b \in F_{+}$ such that $b O \in I({\mathfrak b})$, $b \in O$, $ab \in O$, $ab \equiv b \bmod {\mathfrak b}$. We call the quotient group ${\rm Cl}({\mathfrak b})=I({\mathfrak b})/P_{+}({\mathfrak b})$ the narrow ray class group modulo $\mathfrak b$. This group is known to be finite for any non-zero $\mathfrak b$. An inclusion $\mathfrak b \subset \mathfrak b'$ of integral ideals induces a canonical homomorphism ${\rm Cl}(\mathfrak b) \rightarrow {\rm Cl}(\mathfrak b')$. 
\begin{definition}
\label{defn:201}
A narrow ray class character modulo an integral ideal ${\mathfrak b}$ is a group homomorphism $\psi: {\rm Cl}(\mathfrak b) \rightarrow {\mathbb{C}}^{\times}$. 
\end{definition}
The conductor of a narrow ray class character $\psi$ modulo $\mathfrak b$ is a unique integral ideal $\mathfrak c$ which has the following properties:
\begin{itemize}
\item $\mathfrak b \subset \mathfrak c$;
\item the canonical homomorphism $\pi: {\rm Cl}(\mathfrak b) \rightarrow {\rm Cl}(\mathfrak c)$ factors $\psi$, i.e., there exists a homomorphism $\psi_0: {\rm Cl}(\mathfrak c) \rightarrow {\mathbb{C}}^{\times}$ with $\psi=\psi_0 \circ \pi$;
\item for any integral ideal $\mathfrak c'$ with $\mathfrak c \subsetneq \mathfrak c'$, the canonical homomorphism $\pi: {\rm Cl}(\mathfrak c) \rightarrow {\rm Cl}(\mathfrak c')$ does not factor $\psi$.
\end{itemize}
We write ${\rm cond}(\psi)=\mathfrak c$. If $\mathfrak b=\mathfrak c$, then $\psi$ is said to be primitive modulo $\mathfrak b$. 

It is known that there exists a vector $r \in ({\mathbb{Z}}/2{\mathbb{Z}})^d$ such that
$$\psi(a{O})={\rm sgn}(a)^r \ \text{for all} \ a \in {O} \ \text{with}\  a \equiv 1 \bmod \mathfrak b.$$
We call $r$ the signature of $\psi$. With this assumption, we can define a character $\psi_f : (O/\mathfrak b)^{\times} \rightarrow {\mathbb{C}}^{\times}$ associated to $\psi$ by $\psi_f(a)=\psi(a{O}){\rm sgn}(a)^r$. We will always regard the right-hand side as a character on $(O/{\mathfrak b})^\times$, without any notice. 

When ${\mathfrak b}=O$, we write ${\rm Cl}_F^+$ rather than ${\rm Cl}(O)$, and we call this group the narrow ideal class group of $F$. There is a canonical surjective homomorphism from ${\rm Cl}_F^+$ to the (wide) ideal class group ${\rm Cl}_F$. In particular $h=\# {\rm Cl}_F^+$ is a multiple of the class number of $F$. 

We now describe the definition of (parallel weight) Hilbert modular forms over $F$. First we choose a representative fractional ideal ${\mathfrak t}_\lambda$ of $\lambda$ for each $\lambda \in {\rm Cl}_F^+$, and define the subgroup
$$\Gamma_\lambda({\mathfrak b})=\left\{\left(
\begin{array}{cc}
a & b \\
c & d
\end{array}
\right) \in GL_2^{+}(F) \bigg\vert a, d \in {O}, b \in {\mathfrak d}^{-1}{\mathfrak t}_\lambda^{-1}, c \in {\mathfrak b}{\mathfrak d}{\mathfrak t}_\lambda \ \text{and} \ ad-bc \in {O}^{\times} \right\}$$
of $GL_2^{+}(F)$. 
\begin{definition}
\label{defn:202}
Let $k \geq 0$ be an integer, and $\mathfrak b, \psi$ as above. The space $M_k(\mathfrak b, \psi)$ of Hilbert modular forms of weight $k$, level $\mathfrak b$ and character $\psi$ consists of elements $f$ such that
\begin{itemize}
\item $f=(f_\lambda)_{\lambda \in {\rm Cl}_F^+}$ is an $h$-tuple of holomorphic functions $f_\lambda : {\mathfrak H}^I \rightarrow \mathbb{C}$;
\item for each $\lambda \in {\rm Cl}_F^+$, $f_\lambda$ satisfies the following modularity property:
\begin{align}
 \ \ f_\lambda |_k \gamma &= \psi_f(d)f_\lambda \ \text{for all} \ \gamma=\left(
\begin{array}{cc}
a & b \\
c & d
\end{array}
\right) \in \Gamma_\lambda(\mathfrak b). \label{eqn:205}
\end{align}
Here
$$\ \ \ \det(\gamma)=\prod_{\sigma \in I}\det(\gamma)^{\sigma}, \ cz+d=\prod_{\sigma \in I}(c^{\sigma}z_{\sigma}+d^{\sigma}), \ \gamma z=\left(\frac{a^{\sigma}z_{\sigma}+b^{\sigma}}{c^{\sigma}z_{\sigma}+d^{\sigma}} \right)_{\sigma \in I}$$
and $f_\lambda|_k \gamma$ is a function on ${\mathfrak H}^I$ defined by
\begin{align*}
 \ \ (f_\lambda|_k \gamma)(z) &= \det(\gamma)^{\frac{k}{2}}(cz+d)^{-k}f_\lambda(\gamma z), 
\end{align*}
Since each $f_\lambda$ is a function on ${\mathfrak H}^I$, we regard $z$ a $d$-tuple of variables $z_\sigma$. We also note that $\gamma z \in {\mathfrak H}^I$ for any $\gamma \in GL_2^{+}(F)$. We often omit the subscript $k$ of $f_\lambda|_k \gamma$ when there is no ambiguity concerning weight. 
\item when $F={\mathbb{Q}}$, we also impose the holomorphy condition around each cusp; that is, for any $\gamma \in SL_2({\mathbb{Z}})$, we have
\begin{align*}
 \ \ (f|_k \gamma)(z) &= \sum_{n=0}^\infty a\left(\frac{n}{N}, f \right)q^{\frac{n}{N}} \ \ (q=\exp(2\pi iz))
\end{align*}
where $N$ is the positive integer determined by $N{\mathbb{Z}}={\mathfrak b}$. 
\end{itemize}
\end{definition}
\begin{remark}
\label{rmk:203}
The definition of the subgroup $\Gamma_\lambda({\mathfrak b})$ depends on the choice of a representative fractional ideal ${\mathfrak t}_\lambda$. We take two representative ideals ${\mathfrak t}_{\lambda, i} \ (i=1,2)$ of $\lambda \in {\rm Cl}_F^+$ and consider the ${\mathbb{C}}$-vector space $M_k({\mathfrak b}, \psi)_i$ consisting of modular forms satisfying the modularity property (\ref{eqn:205}) with respect to $\Gamma_{{\mathfrak t}_{\lambda, i}}({\mathfrak b})$ for each $i$. By definition we have ${\mathfrak t}_{\lambda, 2}=u {\mathfrak t}_{\lambda, 1}$ for some $u \in F_+$. Then there is an isomorphism
$$M_k({\mathfrak b}, \psi)_1 \rightarrow M_k({\mathfrak b}, \psi)_2; \ (f_\lambda)_{\lambda \in {\rm Cl}_F^+} \mapsto \left(f_\lambda|_k \left(
\begin{array}{cc}
 u & 0 \\
 0 & 1
\end{array}
\right) \right)_{\lambda \in {\rm Cl}_F^+}. $$
However we can define Fourier coefficients of $f$ independent of the choice of a representative ideal ${\mathfrak t}_\lambda$ (see Definition \ref{defn:205} and Remark \ref{rmk:206} for details).
\end{remark}

We define a Fourier expansion of a Hilbert modular form. 
\begin{proposition}
\label{prop:204}
A Hilbert modular form $f=(f_\lambda)_{\lambda \in {\rm Cl}_F^+} \in M_k(\mathfrak b, \psi)$ has a Fourier expansion $($at the cusp $\infty=(\infty, \infty, \ldots, \infty))$ of the following form$:$
\begin{align}
 f_\lambda(z) &= a_\lambda(0)+\sum_{b \in {\mathfrak t}_\lambda \cap F_{+}} a_\lambda(b)e_F(bz) \ \text{for each} \ \lambda \in {\rm Cl}_F^+. \label{eqn:302}
\end{align}
Here $a_\lambda(0), \ a_\lambda(b)$ are complex numbers and $e_F(x)=\exp(2 \pi i {\rm Tr}(x))=\exp(2 \pi i \sum_{\sigma \in I}x_\sigma)$ $($we use this notation both for $x \in F$ and for a $d$-tuple of variables $x=(x_\sigma)_{\sigma \in I})$. 
\end{proposition}
\begin{proof}
The assertion is well known when $F={\mathbb{Q}}$. When $F \neq {\mathbb{Q}}$, ideas of the proof are basically the same as that for $F={\mathbb{Q}}$. Namely, the modularity property (\ref{eqn:205}) implies that $f_\lambda(z)$ is invariant under the translation by elements of ${\mathfrak t}_\lambda^{-1}{\mathfrak d}^{-1}$, and since $f_\lambda$ is holomorphic in $z$ we conclude that $f$ is of the form
\begin{align*}
 f_\lambda(z) &= \sum_{b \in {\mathfrak t}_\lambda} a_\lambda(b)e_F(bz). 
\end{align*}
We need to show that $a_\lambda(b)=0$ for all $b \in {\mathfrak t}_\lambda$ with $b \notin F_{+}$ and $b \neq 0$. This is so-called ``Koecher's principle" (see \cite{G} Theorem 3.3 of Chapter 2, Section 3). This principle does not hold when $F={\mathbb{Q}}$. 
\end{proof}

We call the coefficients $a_\lambda(b)$ the unnormalized Fourier coefficients of $f$. We also define the normalized one as follows.
\begin{definition}
\label{defn:205}
Let $f$ be as in Proposition \ref{prop:204}. We define the normalized constant term $c_\lambda(0, f)$ of $f$ by $c_\lambda(0, f)=a_\lambda(0){\rm N}({\mathfrak t}_\lambda)^{-\frac{k}{2}}$ for each $\lambda \in {\rm Cl}_F^+$.

For each non-zero integral ideal $\mathfrak n$ of $F$, there exists a unique $\lambda \in {\rm Cl}_F^+$, and $b \in {\mathfrak t}_\lambda \cap F_{+}$ unique up to multiplication by totally positive units, such that ${\mathfrak n}=b {\mathfrak t}_\lambda^{-1}$. We define the normalized Fourier coefficient $c({\mathfrak n}, f)$ associated to $\mathfrak n$ by $c({\mathfrak n}, f)=a_\lambda(b){\rm N}({\mathfrak t}_\lambda)^{-\frac{k}{2}}$. 
\end{definition}

\begin{remark}
\label{rmk:206}
The following two facts show why $c_\lambda(0, f)$ and $c({\mathfrak n}, f)$ are called ``normalized'' coefficients. These facts can be deduced from the modularity property (\ref{eqn:205}). 
\begin{itemize}
\item[(i)] $c_\lambda(0, f)$ and $c({\mathfrak n}, f)$ are independent of the choice of a representative fractional ideal ${\mathfrak t}_\lambda$.
\item[(ii)] $c({\mathfrak n}, f)$ is independent of the choice of $b \in {\mathfrak t}_\lambda \cap F_{+}$.
\end{itemize}
\end{remark}

\subsection{Eisenstein series}
In this subsection we introduce Eisenstein series, which are one of the most basic example of Hilbert modular forms. Let $\eta$ (resp. $\psi$) be a primitive narrow ray class character of conductor $\mathfrak a$ (resp. $\mathfrak b$) and signature $q \in ({\mathbb{Z}}/2{\mathbb{Z}})^d$ (resp. $r$). Actually we can define Eisenstein series for non-primitive characters, but for simplicity, we content ourselves only with primitive case here. When we consider Eisenstein series, we always impose the assumption
\begin{align}
 q+r & \equiv (k, k, \ldots, k) \bmod (2{\mathbb{Z}})^d \label{eqn:308}
\end{align}
for a weight $k$. 
\begin{proposition}[\cite{S} Proposition 3.4]
\label{prop:207}
Under the above condition, there exists a Hilbert modular form $E_k(\eta, \psi)$ of weight $k$, level ${\mathfrak a}{\mathfrak b}$ and character $\eta \psi$ with the following normalized coefficients$:$
\begin{align}
 c({\mathfrak n}, E_k(\eta, \psi)) &= \sum_{{\mathfrak n}_1 \mid {\mathfrak n}} \eta\left(\frac{{\mathfrak n}}{{\mathfrak n}_1} \right) \psi({\mathfrak n}_1) {\rm N}({\mathfrak n}_1)^{k-1} \ \text{for each non-zero integral ideal} \ \mathfrak n; \label{eqn:309} \\
 c_\lambda(0, E_k(\eta, \psi)) &= \left\{
\begin{array}{l}
 \delta_{\eta, {\rm id}}2^{-d}L(\psi, 1-k) \ \mbox{if\ } k \geq 2, \\
 2^{-d}(\delta_{\eta, {\rm id}}L(\psi, 0)+\delta_{\psi, {\rm id}}L(\eta, 0)) \ \mbox{if\ } k=1 
\end{array}
\right. \ \text{for each} \ \lambda \in {\rm Cl}_F^+. \label{eqn:310}
\end{align}
The sum in $(\ref{eqn:309})$ runs over all integral ideals ${\mathfrak n}_1$ dividing ${\mathfrak n}$. In $(\ref{eqn:310})$, $\delta_{\eta, {\rm id}}=1$ if $\eta={\rm id}$ $($i.e., ${\mathfrak a}=O)$ and $0$ otherwise. $L(\eta, s)$ denotes the Hecke $L$-function attached to the character $\eta$ $($we use the same notation for other characters$)$. We call $E_k(\eta, \psi)$ the Eisenstein series of weight $k$ associated with characters $(\eta, \psi)$.
\end{proposition}
Before explaining the outline of the proof, we establish two lemmas on the ideal class groups of $F$, which will be frequently used afterwards.
\begin{lemma}
\label{lem:208}
Let $\mathfrak a$ and $\mathfrak m$ be non-zero integral ideals of $F$. Then there exists a totally positive element $a \in \mathfrak a$ such that $aO=\mathfrak a \mathfrak n$ with $\mathfrak n$ prime to $\mathfrak m$.
\end{lemma}
\begin{proof}
Let ${\mathfrak m}=\prod_{i=1}^l {\mathfrak p}^{e_i}$ be the prime ideal factorization of $\mathfrak m$. We put ${\mathfrak c}={\mathfrak a} {\mathfrak p}_1 {\mathfrak p}_2 \cdots {\mathfrak p}_l$, and ${\mathfrak c}_i={\mathfrak c}/{\mathfrak p}_i$ for each $1 \leq i \leq l$. Since ${\mathfrak c} \subsetneq {\mathfrak c}_i$, there exists an element $a_i \in {\mathfrak c}_i$ with $a_i \notin \mathfrak c$. We may assume that each $a_i$ is totally positive. Indeed, since the images of $a_iO$ and $a_i{\mathfrak p}_i$ in ${\mathbb{R}}^I$ are ${\mathbb{Z}}$-lattices, there exists $b_i \in O$ for each $i$ such that $a_ib_i$ is totally positive. We may further assume that $b_i \notin {\mathfrak p}_i$, since each $a_i{\mathfrak p}_i$ is a proper sublattice of $a_iO$ in ${\mathbb{R}}^I$. Then the totally positive element $a=a_1+a_2+ \cdots +a_l \in \mathfrak a$ satisfies the assertion, that is, we have $a \notin {\mathfrak a} {\mathfrak p}_i$ for each $i$. Otherwise, we see that $a_i=a-\sum_{j \neq i} a_j \in {\mathfrak a} {\mathfrak p}_i$, which contradicts the choice of $a_i$. Hence we have $aO=\mathfrak a \mathfrak n$ with an integral ideal $\mathfrak n$ prime to $\mathfrak m$.
\end{proof}
In the proof of Lemma \ref{lem:208}, we have implicitly assumed that for any ${\mathbb{Z}}$-lattice $L$ of ${\mathbb{R}}^d$, we can choose a ${\mathbb{Z}}$-basis $v_1, v_2, \ldots, v_d$ of $L$ so that $v_1, v_2, \ldots, v_d \in ({\mathbb{R}}_{>0})^d$. We give a proof of this elementary fact in the appendix. 
\begin{lemma}
\label{lem:209}
Let ${\mathfrak m}$ be a non-zero integral ideal. We can choose a representative set $\left\{{\mathfrak t}_\lambda \right\}$ for ${\rm Cl}_F^+$ so that ${\mathfrak b}{\mathfrak d}{\mathfrak t}_\lambda$ is integral and prime to $\mathfrak m$ for each $\lambda \in {\rm Cl}_F^+$.
\end{lemma}
\begin{proof}
For $\lambda \in {\rm Cl}_F^+$, we choose an integral ideal $\mathfrak c$ belonging to the class of $\lambda$. We decompose ${\mathfrak b}{\mathfrak d}{\mathfrak c}$ as ${\mathfrak b}{\mathfrak d}{\mathfrak c}={\mathfrak c}_0{\mathfrak c}'$ so that the prime factors of ${\mathfrak c}_0$ are all in common with $\mathfrak m$ and ${\mathfrak c}'$ is prime to $\mathfrak m$. We choose a totally positive element $c \in {\mathfrak c}_0$. Then we have $cO={\mathfrak c}_0 \mathfrak n$ for some integral ideal $\mathfrak n$. By Lemma \ref{lem:208}, there exists a totally positive element $a \in \mathfrak n$ so that $aO=\mathfrak n {\mathfrak n}'$ with ${\mathfrak n}'$ prime to $\mathfrak m$. Therefore we have
$$\frac{a}{c}{\mathfrak b}{\mathfrak d}{\mathfrak c}=\frac{{\mathfrak n}{\mathfrak n}'{\mathfrak c}_0 {\mathfrak c}'}{{\mathfrak c}_0 \mathfrak n}={\mathfrak n}' {\mathfrak c}',$$
which is integral and prime to ${\mathfrak m}$. Since $a/c$ is totally positive by our choice, the assertion holds. 
\end{proof}
\begin{remark}
\label{rmk:210}
Lemmas \ref{lem:208} and \ref{lem:209} are originally Exercise 2.5.2 of \cite{H1}. The answer given at the end of the book is not enough. He did not mention whether the element $a/c$ in the proof of Lemma \ref{lem:209} can be taken to be totally positive or not. 
\end{remark}
Hereafter we write ${\mathfrak m}={\mathfrak a}{\mathfrak b}$. As a consequence of Lemma \ref{lem:209}, we may and do assume the following two conditions:
\begin{itemize}
\item for each $\lambda \in {\rm Cl}_F^+$, ${\mathfrak b}{\mathfrak d}{\mathfrak t}_\lambda$ is integral and prime to $\mathfrak m$;
\item each representative fractional ideal $\mathfrak r$ of ${\rm Cl}_F$ is integral and prime to $\mathfrak m$.
\end{itemize}
\begin{proof}(Outline of the proof of Proposition \ref{prop:207})
The Eisenstein series $E_k(\eta, \psi)$ in Proposition \ref{prop:207} is explicitly given in \cite{S}  Proposition 3.2 and \cite{DDP}  Proposition 2.1. We recall the definition. For $s \in \mathbb{C}$, the series
\begin{align*}
E_k(\eta, \psi)_\lambda(z, s) &= C\tau(\psi)\frac{{\rm N}({\mathfrak t}_\lambda)^{-\frac{k}{2}}}{{\rm N}({\mathfrak b})}\sum_{{\mathfrak r} \in {\rm Cl}_F}{\rm N}({\mathfrak r})^k \sum_{\substack{
 a \in {\mathfrak r}, \\
 b \in ({\mathfrak b}{\mathfrak d}{\mathfrak t}_\lambda)^{-1} {\mathfrak r}, \\
 (a, b) \bmod U, \\
 (a, b) \neq (0, 0)
}}
\frac{{\rm sgn}(a)^q \eta(a{\mathfrak r}^{-1}){\rm sgn}(-b)^r \psi^{-1}(-b {\mathfrak b}{\mathfrak d}{\mathfrak t}_\lambda{\mathfrak r}^{-1})}{(az+b)^k |az+b|^{2s}} 
\end{align*}
is convergent on the right half plane ${\rm Re}(k+2s)>2$. Here
\begin{align*}
 \tau(\psi) &= \sum_{x \in {\mathfrak b}^{-1}{\mathfrak d}^{-1}/{\mathfrak d}^{-1}} {\rm sgn}(x)^r \psi(x {\mathfrak b}{\mathfrak d})e_F(x)
\end{align*}
is the Gauss sum of $\psi$, $U$ is the subgroup of finite index of $O^{\times}$ defined by
\begin{align*}
 U &= \left\{u \in O^{\times} \mid {\rm N}(u)^k=1, u \equiv 1 \bmod \mathfrak m \right\}
\end{align*}
which acts on $\left\{(a, b) \mid a \in {\mathfrak r}, \ b \in ({\mathfrak b}{\mathfrak d}{\mathfrak t}_\lambda)^{-1} {\mathfrak r}, (a, b) \neq (0, 0) \right\}$ by $u \cdot (a, b)=(ua, ub)$, and
\begin{align*}
 C &= \frac{\sqrt{d_F}\Gamma(k)^d}{[O^{\times} : U]{\rm N}({\mathfrak d}) (-2\pi i)^{kd}}
\end{align*}
where $d_F$ denotes the discriminant of $F$. The definition of $E_k(\eta, \psi)_\lambda(z, s)$ here looks slightly different from that in \cite{DDP}, but in fact two definitions are exactly the same. We have already computed some terms of $E_k(\eta, \psi)_\lambda(z, s)$ in \cite{DDP} by using the hypothesis that $\psi$ is primitive.

It is enough to show that $E_k(\eta, \psi)_\lambda(z, s)$ has a meromorphic continuation in $s$ to the whole complex plane and is holomorphic at $s=0$, and that the $h$-tuple $(E_k(\eta, \psi)_\lambda(z, 0))_{\lambda \in {\rm Cl}_F^+}$ is a Hilbert modular form of prescribed weight, level and character, with the desired Fourier coefficients. Roughly speaking, our strategy is as follows: thanks to Lemma \ref{lem:209}, we can divide $E_k(\eta, \psi)_\lambda(z, s)$ into partial sums
\begin{align*}
 E_k(\eta, \psi)_\lambda(z, s) &= C\tau(\psi)\frac{{\rm N}({\mathfrak t}_\lambda)^{-\frac{k}{2}}}{{\rm N}({\mathfrak b})} \sum_{{\mathfrak r} \in {\rm Cl}_F}{\rm N}({\mathfrak r})^k \sum_{a_1 \in \mathfrak r/\mathfrak r \mathfrak m}{\rm sgn}(a_1)^q \eta(a_1 {\mathfrak r}^{-1}) \nonumber \\
 & \ \times \sum_{a_2 \in ({\mathfrak b}{\mathfrak d}{\mathfrak t}_\lambda)^{-1}{\mathfrak r}/{\mathfrak m} ({\mathfrak b}{\mathfrak d}{\mathfrak t}_\lambda)^{-1}{\mathfrak r}}{\rm sgn}(-a_2)^r \psi^{-1}(-a_2 {\mathfrak b}{\mathfrak d}{\mathfrak t}_\lambda{\mathfrak r}^{-1})G_k(z, s; a_1, a_2, \mathfrak m, \mathfrak r)
\end{align*}
where we define
\begin{equation*}
 G_k(z, s; a_1, a_2, {\mathfrak m}, {\mathfrak r}) = \sum_{\substack{
 a \in {\mathfrak r}, \ b \in ({\mathfrak b}{\mathfrak d}{\mathfrak t}_\lambda)^{-1}{\mathfrak r}, \\
 a-a_1 \in {\mathfrak m}{\mathfrak r}, \ b-a_2 \in {\mathfrak m}({\mathfrak b}{\mathfrak d}{\mathfrak t}_\lambda)^{-1}{\mathfrak r}, \\
 (a, b) \bmod U, \ (a, b) \neq (0, 0) 
}}
\frac{1}{(az+b)^k|az+b|^{2s}} 
\end{equation*}
for ${\mathfrak r} \in {\rm Cl}_F, \ a_1 \in \mathfrak r$, and $a_2 \in ({\mathfrak b}{\mathfrak d}{\mathfrak t}_\lambda)^{-1}\mathfrak r$. Indeed, if $a \in \mathfrak r$ and $b \in ({\mathfrak b}{\mathfrak d}{\mathfrak t}_\lambda)^{-1}\mathfrak r$ satisfy $a-a_1 \in \mathfrak m \mathfrak r$ and $b-a_2 \in {\mathfrak m}({\mathfrak b}{\mathfrak d}{\mathfrak t}_\lambda)^{-1}\mathfrak r$ respectively, then 
\begin{align}
 {\rm sgn}(a)^q\eta(a)={\rm sgn}(a_1)^q\eta(a_1) \ \ & \text{and} \ \ {\rm sgn}(-b)^r\psi^{-1}(b)={\rm sgn}(-a_2)^r\psi^{-1}(a_2). \label{eqn:312}
\end{align}
We need to make a remark on the latter equality. By Lemma \ref{lem:208}, there exists an element $\beta \in {\mathfrak b}{\mathfrak d}{\mathfrak t}_\lambda$ with $\beta O={\mathfrak b}{\mathfrak d}{\mathfrak t}_\lambda \mathfrak n$ and $\mathfrak n$ prime to $\mathfrak b$. We note that $\beta$ itself is prime to $\mathfrak b$ since ${\mathfrak b}{\mathfrak d}{\mathfrak t}_\lambda$ is prime to ${\mathfrak m}$. Then $\beta b, \beta a_2 \in {\mathfrak r}{\mathfrak n}$ and $\beta(b-a_2) \in {\mathfrak m}{\mathfrak r}{\mathfrak n}$ so that
\begin{align}
 {\rm sgn}(-\beta b)^r\psi^{-1}(\beta b) &= {\rm sgn}(-\beta a_2)^r\psi^{-1}(\beta a_2). \label{eqn:314}
\end{align}
Dividing both sides of (\ref{eqn:314}) by ${\rm sgn}(\beta)^r \psi^{-1}(\beta) \neq 0$ gives the equality (\ref{eqn:312}). We can verify Proposition \ref{prop:207} by computing the Fourier expansion of each $G_k(z, s; a_1, a_2, {\mathfrak m}, {\mathfrak r})$. 
\end{proof}

%%%%%%%%%%%%%%%%%%%%%%%%%%%%%%%%%%%%%%%%%%%

\section{Equivalence classes of cusps and the main theorem}
Section 3 is devoted a formulation and the proof of our main theorem. We keep using the notation at the beginning of Section 2. 
\subsection{Constant terms of Eisenstein series under slash operators}
In this subsection, we present a detailed computation of the normalized constant term of $E_k(\eta, \psi)$ under the slash operators defined below. First we introduce some congruence subgroups of $GL_2^{+}(F)$ and $SL_2(F)$. The notation in the following definition is basically in accordance with \cite{H2} Chapter 4, Section 1.3.
\begin{definition}
\label{defn:301}
Let ${\mathfrak n}$ be an integral ideal and ${\mathfrak c}$ a fractional ideal of $F$. $\Gamma_1({\mathfrak n}; O, {\mathfrak c})$ is the subgroup of $GL_2^{+}(F)$ defined by
\begin{align*}
 \Gamma({\mathfrak n}; O, {\mathfrak c}) &= \left\{\left(
\begin{array}{cc}
 a & b \\
 c & d
\end{array}
\right) \in GL_2^+(F) \bigg\vert a, d \in O, b \in {\mathfrak c}^{-1}{\mathfrak d}^{-1}, c \in {\mathfrak n}{\mathfrak c}{\mathfrak d} \ \text{and} \ ad-bc \in O^{\times} \right\}.
\end{align*}
Hereafter we mainly consider the subgroup $\Gamma^1({\mathfrak n}; O, {\mathfrak c})=SL_2(F) \cap \Gamma({\mathfrak n}; O, {\mathfrak c})$ of $SL_2(F)$. 
\end{definition}
\begin{remark}
\label{rmk:302}
\begin{itemize}
\item[(i)] When $F=\mathbb{Q}$, ${\mathfrak n}=N{\mathbb{Z}} \ (N \in {\mathbb{Z}}_{>0})$ and ${\mathfrak c}={\mathbb{Z}}$, we have 
$$\Gamma({\mathfrak n}; O, {\mathfrak c})=\Gamma^1({\mathfrak n}; O, {\mathfrak c})=\Gamma_0(N).$$
\item[(ii)] When ${\mathfrak n}={\mathfrak b}$ and ${\mathfrak c}={\mathfrak t}_\lambda$ for $\lambda \in {\rm Cl}_F^+$, $\Gamma({\mathfrak b}; O, {\mathfrak t}_\lambda)=\Gamma_\lambda({\mathfrak b})$, which was defined just before Definition \ref{defn:202}.
\end{itemize}
\end{remark}
From now on, we write $\Gamma_\lambda^*({\mathfrak n})$ for $\Gamma^*({\mathfrak n}; O, {\mathfrak t}_\lambda)$ ($*=1$ or empty). We define the slash operator on the space of Hilbert modular forms. 
\begin{definition}
\label{defn:303}
Let $f=(f_\lambda)_{\lambda \in {\rm Cl}_F^+}$ be a Hilbert modular form and $A=(A_\lambda)_{\lambda \in {\rm Cl}_F^+} \in SL_2(F)^h$ an $h$-tuple of matrices. The slash operator is defined by
\begin{align}
 f|A & = (f_\lambda|A_\lambda)_{\lambda \in {\rm Cl}_F^+}. \label{eqn:329}
\end{align}
\end{definition}
The main result of this subsection is as follows:
\begin{proposition}
\label{prop:304}
Let $\eta, \psi, k$ be as in Section $3.2 \ ($in particular satisfying $(\ref{eqn:308}))$, and $A=(A_\lambda)_{\lambda \in {\rm Cl}_F^+}$ a slash operator with
\begin{align*}
 A_\lambda &= \left(
\begin{array}{cc}
 \alpha_\lambda & \beta_\lambda \\
 \gamma_\lambda & \delta_\lambda
\end{array}
\right) \in \Gamma_\lambda^1(O)
\end{align*}
for each $\lambda \in {\rm Cl}_F^+$. Then
\begin{align*}
 c_\lambda(0, E_k(\eta, \psi)|A) &= 0
\end{align*}
unless $\gamma_\lambda \in {\mathfrak b}{\mathfrak d}{\mathfrak t}_\lambda$. If this is the case, we have
\begin{align*}
 c_\lambda(0, E_k(\eta, \psi)|A) &= \frac{1}{2^d}\frac{\tau(\eta \psi^{-1})}{\tau(\psi^{-1})} \left(\frac{{\rm N}({\mathfrak b})}{{\rm N}({\mathfrak c})} \right)^k {\rm  sgn}(-\gamma_\lambda)^q \eta(\gamma_\lambda ({\mathfrak b}{\mathfrak d}{\mathfrak t}_\lambda)^{-1}) {\rm sgn}(\alpha_\lambda)^r \psi^{-1}(\alpha_\lambda) \\
 & \ \ \times L(\eta^{-1} \psi, 1-k) \prod_{{\mathfrak q} \mid {\mathfrak m}, \ {\mathfrak q} \nmid {\mathfrak c}} (1-\eta \psi^{-1}({\mathfrak q}){\rm N}({\mathfrak q})^{-k}), 
\end{align*}
where ${\mathfrak c}$ is the conductor of $\eta^{-1} \psi$ and the last product runs over all prime ideals ${\mathfrak q}$ with ${\mathfrak q} \mid {\mathfrak m}$ and ${\mathfrak q} \nmid {\mathfrak c}$.
\end{proposition}
\begin{remark}
\label{rmk:305}
Here we make two remarks on previously known results. 
\begin{itemize}
\item[(i)] Ohta computed the constant terms of Eisenstein series of weight $2$ and level $\Gamma_1(Np^r)$ over ${\mathbb{Q}}$, at all equivalence classes of cusps (Proposition 2.5.5 of \cite{O}). Here $p \geq 5$ is a prime number, $N$ is a positive integer prime to $p$, and $r \geq 1$ is an integer. Proposition \ref{prop:304} is a generalization of his result. Indeed we have $\Gamma_\lambda^1(O)=SL_2(\mathbb{Z})$ when $F={\mathbb{Q}}$ and the condition $\gamma_\lambda \in {\mathfrak b}{\mathfrak d}{\mathfrak t}_\lambda$ here corresponds to $u \mid c$ in \cite{O}. 
\item[(ii)] This proposition also implies Proposition 2.3 of \cite{DDP}, where Dasgupta, Darmon and Pollack computed $c_\lambda(0, E_k(\eta, \psi)|A)$ for 
\begin{align*}
 A = (A_\lambda)_{\lambda \in {\rm Cl}_F^+} &= \left(\left(
\begin{array}{cc}
 1& x_\lambda \\
 \alpha_\lambda & y_\lambda
\end{array} \right)\right)_{\lambda \in {\rm Cl}_F^+} \in SL_2(F)^h
\end{align*}
with $\alpha_\lambda \in {\mathfrak p}{\mathfrak d}{\mathfrak t}_\lambda, \ x_\lambda \in {\mathfrak d}^{-1}{\mathfrak t}_\lambda^{-1}, \ \text{and} \ y_\lambda \in {\mathfrak n}$. Moreover they imposed the hypothesis ${\mathfrak a}{\mathfrak b}={\mathfrak n}{\mathfrak p}$ (here $\mathfrak p$ is a prime ideal and $\mathfrak n$ is an integral ideal prime to $\mathfrak p$). 
\end{itemize}
\end{remark}
\begin{proof}
Hereafter we fix $\lambda \in {\rm Cl}_F^+$. We write down $(E_k(\eta, \psi)_\lambda|A_\lambda)(z, s)$ according to the definition:
\begin{align*}
 (E_k(\eta, \psi)_\lambda|A_\lambda)(z, s) &= C\tau(\psi)\frac{{\rm N}({\mathfrak t}_\lambda)^{-\frac{k}{2}}}{{\rm N}({\mathfrak b})} \sum_{{\mathfrak r} \in {\rm Cl}_F}{\rm N}({\mathfrak r})^k  \\
 & \ \times \sum_{\substack{
 a \in {\mathfrak r}, \\
 b \in ({\mathfrak b}{\mathfrak d}{\mathfrak t}_\lambda)^{-1} {\mathfrak r}, \\
 (a, b) \bmod U, \\
 (a, b) \neq (0, 0)
}} \frac{{\rm sgn}(a)^q \eta(a {\mathfrak r}^{-1}) {\rm sgn}(-b)^r \psi^{-1}(-b {\mathfrak b}{\mathfrak d}{\mathfrak t}_\lambda{\mathfrak r}^{-1}) |\gamma_\lambda z+\delta_\lambda|^{2s}}{((a\alpha_\lambda+b\gamma_\lambda)z+(a\beta_\lambda +b\delta_\lambda))^k |(a\alpha_\lambda+b\gamma_\lambda)z+(a\beta_\lambda +b\delta_\lambda)|^{2s}}.
\end{align*}
We note that the constant term arises from terms with $a\alpha_\lambda+b\gamma_\lambda=0$ and this implies $b\gamma_\lambda=-a\alpha_\lambda \in {\mathfrak r}$. On the other hand the condition $\gamma_\lambda \in {\mathfrak d}{\mathfrak t}_\lambda$ implies that there exists an integral ideal $\mathfrak n$ with $\gamma_\lambda O={\mathfrak n}{\mathfrak d}{\mathfrak t}_\lambda$ and hence $b\gamma_\lambda \in {\mathfrak n}{\mathfrak b}^{-1}{\mathfrak r}$. Our strategy is to focus on the ideal $({\mathfrak n}{\mathfrak b}^{-1}{\mathfrak r}) \cap {\mathfrak r}=({\mathfrak n} \cap {\mathfrak b}){\mathfrak b}^{-1}{\mathfrak r}$. We divide the argument into two cases: 
\begin{itemize}
\item[Case 1:] ${\mathfrak b} \nmid {\mathfrak n}$. Then there exists a prime factor ${\mathfrak p}$ of ${\mathfrak b}$ which satisfies
\begin{align*}
 {\mathfrak n} &= {\mathfrak p}^e {\mathfrak n}', \ {\mathfrak b}={\mathfrak p}^f {\mathfrak b}' \ (e \in {\mathbb{Z}}_{\geq 0}, \ f \in {\mathbb{Z}}_{>0}, \ {\mathfrak p} \nmid {\mathfrak n}'{\mathfrak b}') \ \text{and} \ e<f.
\end{align*}
Then $b\gamma_\lambda \in {\mathfrak n}{\mathfrak b}^{-1}{\mathfrak r} \cap {\mathfrak r}=({{\mathfrak b}'})^{-1}({\mathfrak n}' \cap {\mathfrak b}'){\mathfrak r}$. Thus $b \in {\mathfrak n}^{-1}({{\mathfrak b}'})^{-1}({\mathfrak n}' \cap {\mathfrak b}'){\mathfrak d}^{-1}{\mathfrak t}_\lambda^{-1}{\mathfrak r}$ and $b {\mathfrak b}{\mathfrak d}{\mathfrak t}_\lambda{\mathfrak r}^{-1} \subset {\mathfrak p}^{f-e}({{\mathfrak n}'})^{-1}({\mathfrak n}' \cap {\mathfrak b}') \subset {\mathfrak p}^{f-e}$. Since $f-e>0$, $b{\mathfrak b}{\mathfrak d}{\mathfrak t}_\lambda{\mathfrak r}^{-1}$ is not prime to ${\mathfrak b}$ and thus ${\rm sgn}(-b)^r \psi^{-1}(-b{\mathfrak b}{\mathfrak d}{\mathfrak t}_\lambda{\mathfrak r}^{-1})=0$. 
\item[Case 2:] $\mathfrak b \mid \mathfrak n$. In this case, we know that $\gamma_\lambda \in {\mathfrak b}{\mathfrak d}{\mathfrak t}_\lambda$. Then the matrix $A_\lambda$ induces an isomorphism
\begin{align*}
 \left\{(a, b) \mid a \in {\mathfrak r}, \ b \in ({\mathfrak b}{\mathfrak d}{\mathfrak t}_\lambda)^{-1}{\mathfrak r}, \ a\alpha_\lambda+b\gamma_\lambda=0 \right\}/U & \rightarrow ({\mathfrak b}{\mathfrak d}{\mathfrak t}_\lambda)^{-1}{\mathfrak r}/U; \nonumber \\
 (a, b) & \mapsto a\beta_\lambda+b\delta_\lambda 
\end{align*}
(the inverse map is given by $d \mapsto (-d\gamma_\lambda, d\alpha_\lambda)$). Then
\begin{align}
 & c_\lambda(0, E_k(\eta, \psi)|A) \nonumber \\
 &= C\tau(\psi)\frac{{\rm N}({\mathfrak t}_\lambda)^{-k}}{{\rm N}({\mathfrak b})} \sum_{{\mathfrak r} \in {\rm Cl}_F}{\rm N}({\mathfrak r})^k \nonumber \\
 & \ \times \sum_{\substack{
 d \in ({\mathfrak b}{\mathfrak d}{\mathfrak t}_\lambda)^{-1} {\mathfrak r}, \\
 d \bmod U, \ d \neq 0
}} {\rm sgn}(-d\gamma_\lambda)^q \eta(-d\gamma_\lambda {\mathfrak r}^{-1}) {\rm sgn}(-d\alpha_\lambda)^r \psi^{-1}(-d\alpha_\lambda {\mathfrak b}{\mathfrak d}{\mathfrak t}_\lambda{\mathfrak r}^{-1}) {\rm N}(b)^{-k} \nonumber \\
 &= C\frac{\tau(\psi){\rm N}({\mathfrak b}{\mathfrak d})^k {\rm sgn}(\gamma_\lambda)^q\eta(\gamma_\lambda){\rm sgn}(\alpha_\lambda)^r\psi^{-1}(\alpha_\lambda)[O^{\times} : U]}{(-1)^{kd}{\rm N}({\mathfrak b})\eta({\mathfrak b}{\mathfrak d}{\mathfrak t}_\lambda)} \nonumber \\
 & \ \times L(\eta\psi^{-1}, k) \prod_{{\mathfrak q} \mid {\mathfrak m}, \ {\mathfrak q} \nmid {\mathfrak c}} (1-\eta\psi^{-1}({\mathfrak q}){\rm N}({\mathfrak q})^{-k}). \label{eqn:332}
\end{align}
We use the functional equation for $L(\eta\psi^{-1}, s)$ (see \cite{M} Chapter 3, Section 3):
\begin{align}
 \frac{|d_F|^{\frac{1}{2}-k}{\rm N}({\mathfrak c})^{1-k}(2\pi i)^{kd}}
{2^d\Gamma(k)^d\tau(\eta^{-1}\psi)} L(\eta^{-1}\psi, 1-k) &= L(\eta\psi^{-1}, k). \label{eqn:333}
\end{align}
We obtain the desired result by combining equalities (\ref{eqn:332}) and (\ref{eqn:333}). 
\end{itemize}
\end{proof}

\subsection{The equivalence classes of cusps of congruence subgroups}
The purpose of this subsection is to investigate the equivalence classes of cusps by the action of the subgroup $\Gamma_\lambda^1(O)$. First we describe the set of cusps ${\mathbb{P}}^1(F)$ of ${\mathfrak H}^I$ in terms of a quotient of $SL_2(F)$. Let $B^+(F)$ denote the subgroup of $GL_2^+(F)$ consisting of all upper triangular matrices in $GL_2^+(F)$, and $B^1(F)=B^+(F) \cap SL_2(F)$ its intersection with $SL_2(F)$. The following bijection is well known. 
\begin{lemma}
\label{lem:306}
There is a bijection
\begin{align*}
 SL_2(F)/B^1(F) & \rightarrow {\mathbb{P}}^1(F) ; \ \gamma \mapsto \gamma(\infty).
\end{align*}
\end{lemma}
Let ${\mathfrak c}$ be a fractional ideal of $F$. Thanks to Lemma \ref{lem:306} we know that the equivalence classes of cusps by the action of $\Gamma(O; O, {\mathfrak c}^{-1})$ is 
$$\Gamma(O; O, {\mathfrak c}^{-1}) \backslash SL_2(F)/B^1(F).$$ 
We now describe this set explicitly (here we consider $\Gamma(O; O, {\mathfrak c}^{-1})$ instead of $\Gamma(O; O, {\mathfrak c})$, in order to be consistent with the notation used in \cite{H2} Chapter 4, Section 1). To $m=\left(
\begin{smallmatrix}
 a & * \\
 c & *
\end{smallmatrix}
\right) \in SL_2(F)$, we associate a fractional ideal $il_{\mathfrak c}(m)=c{\mathfrak c}{\mathfrak d}^{-1}+aO$. If $\gamma=\left(
\begin{smallmatrix}
 e & f \\
 g & h
\end{smallmatrix}
\right)$ is an element of $\Gamma(O; O, {\mathfrak c}^{-1})$, we have $il_{\mathfrak c}(\gamma m)=il_{\mathfrak c}(m)$. Indeed, since we have $e, h \in O, \ f \in {\mathfrak c}{\mathfrak d}^{-1}$ and $g \in {\mathfrak c}^{-1}{\mathfrak d}$ by definition, we see that 
$$(ga+hc){\mathfrak c}{\mathfrak d}^{-1} \subset aO+c{\mathfrak c}{\mathfrak d}^{-1} \ \text{and} \ ea+fc \in aO+c{\mathfrak c}{\mathfrak d}^{-1}.$$
This implies $(ga+hc){\mathfrak c}{\mathfrak d}^{-1}+(ea+fc)O \subset c{\mathfrak c}{\mathfrak d}^{-1}+aO$. The other inclusion follows from the condition $eh-fg=1$, that is, we have
\begin{align*}
 a = a(eh-fg) &= h(ea+fc)-f(ga+hc) \in (ga+hc){\mathfrak c}{\mathfrak d}^{-1}+(ea+fc)O, \ \text{and} \\
 c = c(eh-fg) &= -f(ea+fc)+e(ga+hc) \in (ga+hc){\mathfrak c}{\mathfrak d}^{-1}+(ea+fc)O.
\end{align*}
For upper triangular $b=\left(
\begin{smallmatrix}
 \widetilde{b} & * \\
 0 & *
\end{smallmatrix}
\right) \in B^1(F)$, we have $il_{\mathfrak c}(gb)=\widetilde{b} \cdot il_{\mathfrak c}(g)$. Hence we obtain a map
\begin{align}
 il_{\mathfrak c} : \Gamma(O; O, {\mathfrak c}^{-1}) \backslash SL_2(F)/B^1(F) \rightarrow {\rm Cl}_F & ; \ \left(
\begin{array}{cc}
 a & * \\
 c & * 
\end{array}
\right) \mapsto c{\mathfrak c}{\mathfrak d}^{-1}+aO. \label{eqn:336}
\end{align}
\begin{remark}
\label{rmk:307}
In \cite{H2} Chapter 4, Section 1, the map $il_{\mathfrak c}$ is defined by $il_{\mathfrak c}(m)=c{\mathfrak c}{\mathfrak d}+aO$. However this map is not trivial on $\Gamma(O; O, {\mathfrak c}^{-1})$ in general. Indeed, the class number of the real quadratic field $F={\mathbb{Q}}(\sqrt{10})$ is $2$, and ${\mathfrak p}=(2, \sqrt{10})$ is a non-principal prime ideal above $2$. The discriminant is $d_F=40$, the different is ${\mathfrak d}=(2\sqrt{10})=(5, \sqrt{10}){\mathfrak p}^3$. We let ${\mathfrak c}=O$. We have $(2+\sqrt{10})O=(3, \sqrt{10}-1){\mathfrak p}$ and $2\sqrt{10}{\mathfrak d}^{-1}+(2+\sqrt{10})O=O$. For example
\begin{align*}
 \left(
\begin{array}{cc}
 2+\sqrt{10} & \frac{10+\sqrt{10}}{20} \\
 2\sqrt{10} & 1
\end{array}
\right) & \in \Gamma(O; O, O). 
\end{align*}
On the other hand, we see that 
$$2\sqrt{10}{\mathfrak d}+(2+\sqrt{10})O={\mathfrak p}^6(5, \sqrt{10})^2+(3, \sqrt{10}-1){\mathfrak p}={\mathfrak p}({\mathfrak p}^5(5, \sqrt{10})^2+(3, \sqrt{10}-1))={\mathfrak p}, $$
which is not principal. One can also find the description of the map $il_{\mathfrak c}$ in Chapter 2, Section 2 of \cite{G} (however the notation there is different from ours). 
\end{remark}                            
\begin{proposition}[\cite{G} Proposition 2.22]
\label{prop:308}
The above map  $(\ref{eqn:336})$ is a bijection.
\end{proposition}
%\begin{proof}
One can find a detailed proof of the proposition in \cite{G}. However we need a slightly refined version of the surjectivity of $il_{\mathfrak c}$ later on so as to compute the constant terms, so we will review the proof of the surjectivity in Proposition \ref{prop:309}. 

Now we apply Proposition \ref{prop:308} for ${\mathfrak c}={\mathfrak t}_\lambda^{-1}$. Note that $\Gamma_\lambda^1(O)=\Gamma(O; O, {\mathfrak t}_\lambda)$ by definition. In the light of Proposition \ref{prop:308}, what we have computed in the previous subsection is a constant term of $E_k(\eta, \psi)$ at \emph{one} equivalence class of cusps of $\Gamma_\lambda^1(O)$, that is, the equivalence class of $\infty$. We will compute the constant terms at \emph{all} equivalence classes of cusps of $\Gamma_\lambda^1(O)$ in the next subsection. 

%For further information about the equivalence classes of cusps of congruence subgroups, see \cite{H2} Chapter 4, Section 1 for example. 

\subsection{Constant terms of Eisenstein series under slash operators I\hspace{-.1em}I}
Hereafter we fix $\lambda \in {\rm Cl}_F^+$. As declared at the end of the previous subsection, we compute the constant terms of $E_k(\eta, \psi)$ at all equivalence classes of cusps of $\Gamma_\lambda^1(O)$. We choose an element in ${\rm Cl}_F$ and fix its representative integral ideal ${\mathfrak r}_0$. We may assume that ${\mathfrak r}_0$ is prime to ${\mathfrak m}$. We shall prove a slightly refined version of the surjectivity of the map (\ref{eqn:336}), with ${\mathfrak c}={\mathfrak t}_\lambda^{-1}$. 
\begin{proposition}
\label{prop:309}
We can choose a matrix
\begin{align*}
 A_\lambda &= \left(
\begin{array}{cc}
 \alpha_\lambda & \beta_\lambda \\
 \gamma_\lambda & \delta_\lambda
\end{array}
\right) \in SL_2(F)
\end{align*}
with $il_{{\mathfrak t}_\lambda^{-1}}(A_\lambda)={\mathfrak r}_0$ so that
$$\alpha_\lambda O={\mathfrak n}_2{\mathfrak r}_0, \ \beta_\lambda \in ({\mathfrak d}{\mathfrak t}_\lambda{\mathfrak r}_0)^{-1}, \ \gamma_\lambda O={\mathfrak n}_1{\mathfrak d}{\mathfrak t}_\lambda{\mathfrak r}_0 \ \text{and} \ \delta_\lambda \in {\mathfrak r}_0^{-1}.$$
Here ${\mathfrak n}_i \ (i=1, 2)$ are integral ideals such that ${\mathfrak n}_2$ is prime to ${\mathfrak n}_1$. Furthermore, the ideal ${\mathfrak n}_1$ can be chosen so that ${\mathfrak n}_1$ is prime to ${\mathfrak b}$. 
\end{proposition}
\begin{proof}
Let ${\mathfrak r}_0$ be as above and ${\mathfrak b}=\prod_{i=1}^l {\mathfrak p}_i^{e_i}$ the prime ideal factorization of ${\mathfrak b}$. We can take a non-zero element $\gamma_\lambda \in {\mathfrak d}{\mathfrak t}_\lambda{\mathfrak r}_0$ so that $\gamma_\lambda \notin {\mathfrak p}_i{\mathfrak d}{\mathfrak t}_\lambda{\mathfrak r}_0$ for all $i=1, 2, \ldots, l$. This can be proved by an argument similar to that of Lemma \ref{lem:208}, as follows: we let ${\mathfrak c}={\mathfrak p}_1{\mathfrak p}_2 \cdots {\mathfrak p}_l{\mathfrak d}{\mathfrak t}_\lambda{\mathfrak r}_0$ and ${\mathfrak c}_i={\mathfrak c}{\mathfrak p}_i^{-1}$ for each $i=1, 2, \ldots, l$. Since ${\mathfrak c} \subsetneq {\mathfrak c}_i$ there exists $c_i \in {\mathfrak c}_i \setminus {\mathfrak c}$ for each $i$. Then $\gamma_\lambda=c_1+c_2+ \cdots+c_l$ does the job. Note that the ideal ${\mathfrak d}{\mathfrak t}_\lambda{\mathfrak r}_0$ is not necessarily integral, but the same argument as that of Lemma \ref{lem:208} works as long as $O$ is a Dedekind domain. We then write $\gamma_\lambda O={\mathfrak n}_1{\mathfrak d}{\mathfrak t}_\lambda{\mathfrak r}_0$ with ${\mathfrak n}_1$ integral and prime to ${\mathfrak b}$. Lemma \ref{lem:208} implies that there exists an element $\alpha_\lambda \in {\mathfrak r}_0$ such that $\alpha_\lambda O={\mathfrak n}_2{\mathfrak r}_0$ with ${\mathfrak n}_2$ integral and prime to ${\mathfrak n}_1$. Then we have $\gamma_\lambda({\mathfrak d}{\mathfrak t}_\lambda)^{-1}+\alpha_\lambda O={\mathfrak n}_1{\mathfrak r}_0+{\mathfrak n}_2{\mathfrak r}_0={\mathfrak r}_0$. Since this condition is equivalent to $\gamma_\lambda({\mathfrak d}{\mathfrak t}_\lambda {\mathfrak r}_0)^{-1}+\alpha_\lambda{\mathfrak r}_0^{-1}=O$, there exist $\beta_\lambda \in ({\mathfrak d}{\mathfrak t}_\lambda {\mathfrak r}_0)^{-1}$ and $\delta_\lambda \in {\mathfrak r}_0^{-1}$ such that $\alpha_\lambda \delta_\lambda-\beta_\lambda \gamma_\lambda=1$. This proves 
$$A_\lambda=\left(
\begin{array}{cc}
 \alpha_\lambda & \beta_\lambda \\
 \gamma_\lambda & \delta_\lambda
\end{array}
\right) \in SL_2(F) \ \ \text{and} \ \ il_{{\mathfrak t}_\lambda^{-1}}(A_\lambda)={\mathfrak r}_0. $$ 
\end{proof}

Proposition \ref{prop:308} tells us that it suffices to compute the constant term of $E_k(\eta, \psi)_\lambda|A_\lambda$ for $A_\lambda$ in Proposition \ref{prop:309}. We recall the definition of $(E_k(\eta, \psi)_\lambda|A_\lambda)(z, s)$:
\begin{align*}
 (E_k(\eta, \psi)_\lambda|A_\lambda)(z, s) &= C\tau(\psi)\frac{{\rm N}({\mathfrak t}_\lambda)^{-\frac{k}{2}}}{{\rm N}({\mathfrak b})} \sum_{{\mathfrak r} \in {\rm Cl}_F}{\rm N}({\mathfrak r})^k  \\
 & \times \sum_{\substack{
 a \in {\mathfrak r}, \\
 b \in ({\mathfrak b}{\mathfrak d}{\mathfrak t}_\lambda)^{-1} {\mathfrak r}, \\
 (a, b) \bmod U, \\
 (a, b) \neq (0, 0)
}} \frac{{\rm sgn}(a)^q \eta(a {\mathfrak r}^{-1}) {\rm sgn}(-b)^r \psi^{-1}(-b {\mathfrak b}{\mathfrak d}{\mathfrak t}_\lambda{\mathfrak r}^{-1}) |\gamma_\lambda z+\delta_\lambda|^{2s}}{((a\alpha_\lambda+b\gamma_\lambda)z+(a\beta_\lambda +b\delta_\lambda))^k |(a\alpha_\lambda+b\gamma_\lambda)z+(a\beta_\lambda +b\delta_\lambda)|^{2s}}.
\end{align*}
As in the proof of Proposition \ref{prop:304}, we need to consider terms with $a\alpha_\lambda+b\gamma_\lambda=0$. For each ${\mathfrak r} \in {\rm Cl}_F$, we have $a\alpha_\lambda \in {\mathfrak n}_2{\mathfrak r}_0{\mathfrak r}$ and $b\gamma_\lambda \in {\mathfrak b}^{-1}{\mathfrak n}_1{\mathfrak r}_0{\mathfrak r}$. Noting that ${\mathfrak n}_1$ is prime to ${\mathfrak b}$, we see that $b\gamma_\lambda=-a\alpha_\lambda \in ({\mathfrak n}_2{\mathfrak r}_0{\mathfrak r}) \cap ({\mathfrak b}^{-1}{\mathfrak n}_1{\mathfrak r}_0{\mathfrak r})={\mathfrak n}_1{\mathfrak n}_2{\mathfrak r}_0{\mathfrak r}$ and hence $b \in {\mathfrak n}_2 ({\mathfrak d}{\mathfrak t}_\lambda)^{-1}{\mathfrak r}$. Consequently we have $\psi^{-1}(-b{\mathfrak b}{\mathfrak d}{\mathfrak t}_\lambda{\mathfrak r}^{-1})=0$ unless ${\mathfrak b}=O$. If this is the case, we use an isomorphism
\begin{align*}
 \left\{(a, b) \mid a \in {\mathfrak r}, \ b \in ({\mathfrak d}{\mathfrak t}_\lambda)^{-1}{\mathfrak r}, \ a\alpha_\lambda+b\gamma_\lambda=0 \right\}/U & \rightarrow ({\mathfrak d}{\mathfrak t}_\lambda{\mathfrak r}_0)^{-1}{\mathfrak r}/U; \\
 (a, b) & \mapsto a\beta_\lambda+b\delta_\lambda
\end{align*}
to compute (the inverse map is given by $d \mapsto (-d\gamma_\lambda, d\alpha_\lambda)$). The normalized constant term of $E_k(\eta, \psi)_\lambda|A_\lambda$ is equal to
\begin{align*}
 & C{\rm N}({\mathfrak t}_\lambda)^{-k} \sum_{{\mathfrak r} \in {\rm Cl}_F}{\rm N}({\mathfrak r})^k \sum_{\substack{
 d \in ({\mathfrak d}{\mathfrak t}_\lambda{\mathfrak r}_0)^{-1}{\mathfrak r}, \\
 d \bmod U, \\
 d \neq 0
}} {\rm sgn}(-d\gamma_\lambda)^q \eta(-d\gamma_\lambda {\mathfrak r}^{-1}){\rm N}(d)^{-k} \\
 &= C{\rm N}({\mathfrak d}{\mathfrak r}_0)^k {\rm sgn}(-\gamma_\lambda)^q \eta(\gamma_\lambda ({\mathfrak d}{\mathfrak t}_\lambda{\mathfrak r}_0)^{-1}){\rm N}({\mathfrak d}{\mathfrak t}_\lambda{\mathfrak r}_0)^k \sum_{{\mathfrak r} \in {\rm Cl}_F} \sum_{\substack{
 d \in ({\mathfrak d}{\mathfrak t}_\lambda{\mathfrak r}_0)^{-1}{\mathfrak r}, \\
 d \bmod U, \\
 d \neq 0
}} \eta(-d{\mathfrak d}{\mathfrak t}_\lambda {\mathfrak r}_0{\mathfrak r}^{-1}){\rm N}(d{\mathfrak d}{\mathfrak t}_\lambda {\mathfrak r}_0{\mathfrak r}^{-1})^{-k}.
\end{align*}
Combining this and the functional equation
\begin{align*}
 \frac{|d_F|^{\frac{1}{2}-k}{\rm N}({\mathfrak a})^{1-k}(2\pi i)^{kd}}{2^d \Gamma(k)^d \tau(\eta^{-1})}L(\eta^{-1}, 1-k) &= L(\eta, k) 
\end{align*} 
for $L(\eta, s)$ (see \cite{M}, Chapter 3, Section 3), we see that this value is equal to
\begin{align*}
% & \frac{1}{2^d}\tau(\eta)\left(\frac{{\rm N}({\mathfrak r}_0)}{{\rm N}({\mathfrak a})} \right)^k {\rm sgn}(-\gamma_\lambda)^q \eta(\gamma_\lambda ({\mathfrak d}{\mathfrak t}_\lambda{\mathfrak r}_0)^{-1}) L(\eta^{-1}, 1-k) \\
 & \frac{1}{2^d}\tau(\eta)\left(\frac{{\rm N}({\mathfrak r}_0)}{{\rm N}({\mathfrak a})} \right)^k {\rm sgn}(-\gamma_\lambda)^q \eta({\mathfrak n}_1) L(\eta^{-1}, 1-k).
\end{align*}
We give a summary of our computation as a theorem.
\begin{theorem}
\label{thm:310}
\begin{itemize}
\item[(i)] For a matrix
\begin{align*}
 A_\lambda &= \left(
\begin{array}{cc}
 \alpha_\lambda & \beta_\lambda \\
 \gamma_\lambda & \delta_\lambda
\end{array}
\right) \in \Gamma_\lambda^1(O), 
\end{align*}
we write $\gamma_\lambda O={\mathfrak n}_1{\mathfrak d}{\mathfrak t}_\lambda$. Then the constant term of ${\rm N}({\mathfrak t}_\lambda)^{-\frac{k}{2}}E_k(\eta, \psi)_\lambda|A_\lambda$ is equal to $0$ unless ${\mathfrak b} \mid {\mathfrak n}_1$. If this is the case, the constant term is equal to
\begin{align*}
 & \frac{1}{2^d}\frac{\tau(\eta \psi^{-1})}{\tau(\psi^{-1})} \left(\frac{{\rm N}({\mathfrak b})}{{\rm N}({\mathfrak c})} \right)^k {\rm  sgn}(-\gamma_\lambda)^q \eta(\gamma_\lambda ({\mathfrak b}{\mathfrak d}{\mathfrak t}_\lambda)^{-1}) {\rm sgn}(\alpha_\lambda)^r \psi^{-1}(\alpha_\lambda) \\
 & \ \ \times L(\eta^{-1} \psi, 1-k) \prod_{{\mathfrak q} \mid {\mathfrak m}, \ {\mathfrak q} \nmid {\mathfrak c}} (1-\eta \psi^{-1}({\mathfrak q}){\rm N}({\mathfrak q})^{-k}). 
\end{align*}
\item[(ii)] Let
\begin{align*}
 {\mathfrak r}_0, \  
A_\lambda &= \left(
\begin{array}{cc}
 \alpha_\lambda & \beta_\lambda \\
 \gamma_\lambda & \delta_\lambda
\end{array}
\right) \in SL_2(F), \ {\mathfrak n}_i \ (i=1, 2)
\end{align*}
be as in Proposition $\ref{prop:309}$. Then the constant term of ${\rm N}({\mathfrak t}_\lambda)^{-\frac{k}{2}}E_k(\eta, \psi)_\lambda|A_\lambda$ is
\begin{align*}
  & \delta_{\psi, {\rm id}}\frac{1}{2^d}\tau(\eta)\left(\frac{{\rm N}({\mathfrak r}_0)}{{\rm N}({\mathfrak a})} \right)^k {\rm sgn}(-\gamma_\lambda)^q \eta({\mathfrak n}_1) L(\eta^{-1}, 1-k).
\end{align*}
\end{itemize}
\end{theorem}

\appendix
\section{A lemma on a ${\mathbb{Z}}$-lattice in ${\mathbb{R}}^d$}
In this appendix, we show the following lemma which was assumed in the proof of Lemma \ref{lem:208}. The lemma is quite elementary but we give a proof in order to be precise.  
\begin{lemma}
Let $d \geq 1$ be an integer. For any ${\mathbb{Z}}$-lattice $L$ of ${\mathbb{R}}^d$, we can choose a ${\mathbb{Z}}$-basis $v_1, v_2, \ldots, v_d$ of $L$ so that $v_1, v_2, \ldots, v_d \in ({\mathbb{R}}_{>0})^d$. 
\end{lemma}
\begin{proof}
Let $\Omega$ be a fundamental domain of the lattice $L$. It is well known that ${\rm covol}(L)={\rm vol}(\Omega)$ is independent of the choice of $\Omega$. Since ${\mathbb{R}}^d/L$ is compact, $\Omega$ is bounded and hence there exists a point $y_1 \in {\mathbb{R}}^d$ such that $y_1+\Omega \subset ({\mathbb{R}}_{>0})^d$. We choose a point $x_1 \in (y_1+\Omega) \cap L$ and define the one dimensional subspace $V_1={\mathbb{R}}x_1$ of ${\mathbb{R}}^d$ generated by $x_1$. If $d \geq 2$, there exists a point $y_2 \in {\mathbb{R}}^d$ such that $y_2+\Omega \subset ({\mathbb{R}}_{>0})^d \setminus (({\mathbb{R}}_{>0})^d \cap  V_1)$. We choose a point $x_2 \in (y_2+\Omega) \cap L$. By our choice of $y_2$, $x_1$ and $x_2$ are linearly independent over ${\mathbb{R}}$. Therefore the subspace $V_2={\mathbb{R}}x_1+{\mathbb{R}}x_2$ of ${\mathbb{R}}^d$ is two dimensional over ${\mathbb{R}}$. By iterating this operation, we obtain ${\mathbb{R}}$-linearly independent $d$ elements $x_1, x_2, \ldots, x_d$ of $L$. We put $L'=\bigoplus_{i=1}^d {\mathbb{Z}}x_i$. Obviously $L'$ is a sublattice of $L$. We choose $x_1, x_2, \ldots, x_d \in L$ so that ${\rm covol}(L')$ is minimal among such $x_1, x_2, \ldots, x_d$ (this is possible because $L$ is discrete in ${\mathbb{R}}^d$). We will prove that $L'=L$. 

Suppose not. Then there exists an element $x$ of $L$ not contained in $L'$. We have $x=\sum_{i=1}^d a_i x_i$ for some $a_1, a_2, \ldots, a_d \in {\mathbb{R}}$. By considering a translate of $x$ by an element of $L'$, we may assume that $0 \leq a_i<1$ for all $i$ and $0<a_j<1$ for some $j$. Then $x_1, \ldots, x_{j-1}, x, x_{j+1}, \ldots, x_d$ are linearly independent over ${\mathbb{R}}$. We define the sublattice $L''$ of $L$ by $L''=\left(\bigoplus_{1 \leq i \leq d, \ i \neq j} {\mathbb{Z}} x_i \right) \oplus {\mathbb{Z}}x$. Then we have
\begin{align*}
 \left(
\begin{array}{c}
 x_1 \\
 \vdots \\
 x_{j-1} \\
 x \\
 x_{j+1} \\
 \vdots \\
 x_d
\end{array}
\right) &= \left(
\begin{array}{ccccccc}
 1 &  & & & & &  \\
  & \ddots & & & & & \\
  & & 1 & & & &  \\
a_1 & \ldots & a_{j-1} & a_j & a_{j+1} & \ldots & a_d \\
 & & & & 1 & & \\
 & & & & & \ddots & \\
 & & & & & & 1
\end{array}
\right) \left(
\begin{array}{c}
 x_1 \\
 \vdots \\
 x_{j-1} \\
 x_j \\
 x_{j+1} \\
 \vdots \\
 x_d
\end{array}
\right)
\end{align*}
and thus ${\rm covol}(L'')=a_j{\rm covol}(L')<{\rm covol}(L')$. This contradicts the choice of $x_1, x_2, \ldots, x_d$. Therefore we have $L'=L$. 
\end{proof}

%%%%%%%%%%%% References %%%%%%%%%%%%%
%%
%<Author name> is written as Initial of Given Name, and Family Name.
%<Title> is written in roman letters.
%<Journal name> should be abbreviated according to
% the MR Serials Abbreviations List of Mathematical Reviews:
% (Abbreviations of Names of Serials; http://www.ams.org/mr-database)
%For <Pages>, use en-dash "--" between page numbers.
%%

\end{document}